\documentclass[11pt]{article}
\usepackage{amsmath, amssymb, amsthm}
\usepackage{bbm}
\usepackage{enumitem}
\usepackage{hyperref}
\usepackage{tikz}
\usetikzlibrary{calc}
\usepackage{xcolor}
\usepackage{multirow}
\usepackage{comment}
\RequirePackage[authoryear]{natbib}

\newtheorem{theorem}{Theorem}[section]
\newtheorem{lemma}{Lemma}[section]

\newtheorem{proposition}{Proposition}[section]

\newtheorem{question}{Question}[section]

\newcommand{\N}{\mathbb{N}}

\newcommand{\Z}{\mathbb{Z}}

\newcommand{\Prob}{\mathbb{P}}

\usepackage[left=1in, right=1in, bottom = 1in, top = 1in]{geometry}
\title{Limit shape, avalanches, and stabilisation of abelian sandpiles on comb lattices}
\author{Robin Kaiser, Ecaterina Sava-Huss, Julia Überbacher}
\date{\today}

\begin{document}
\maketitle

\begin{abstract}
We study two aspects of the abelian sandpile model on comb lattices. First, we prove that the infinite-volume limit of the stationary measures is supported on the saturated configuration. We then investigate the shape of avalanches induced by adding a particle at the origin in finite boxes of size $(2n+1)\times(2n+1)$ around the origin. In the stationary distribution on these finite boxes, we show that avalanches reach the boundary along the vertical teeth with probability tending to $1$, while the horizontal spread is of order $\sqrt{n}$. Finally, we establish that the single-source limit shape for abelian sandpiles on the comb lattice agrees with the corresponding limit shapes for the divisible sandpile, internal diffusion-limited aggregation (IDLA), and rotor-router aggregation models. This establishes limit shape universality on the comb lattice.
\end{abstract}

\textbf{2020 Mathematics Subject Classification.} 31E05, 60J05, 60J45, 05C81.\\
\textbf{Keywords.} limit shape, abelian sandpile, infinite volume limit, stabilization, avalanches, comb lattice, odometer function, least action principle.

\section{Introduction}

In 1987, Bak, Tang, and Wiesenfeld \cite{BTW87,BTW88} introduced a lattice model for mass distribution that naturally evolves toward a critical state without the need to fine-tune parameters, thereby exhibiting self-organized criticality. Dhar \cite{D90} later extended the model from lattices to arbitrary graphs. Dhar later named it the abelian sandpile model, emphasizing the commutativity of the toppling dynamics. Since then, the abelian sandpile model has been extensively studied in both the mathematics and physics communities.

Consider a finite, undirected, connected graph $G=(V\cup\{s\},E)$ with a distinguished sink vertex $s$. A sandpile configuration consists of indistinguishable particles placed on the vertices $V$ of $G$. A vertex is called stable if it contains fewer particles than its degree; otherwise, it is unstable and legally topples by sending one particle along each incident edge. A sandpile configuration is stable if every vertex in $V$ is stable. The sink vertex $s$ serves to absorb excess mass. In this framework, one defines a Markov chain on the set of stable sandpile configurations as follows. At each time step, a vertex in $V$ is selected uniformly at random and a particle is added to it. If this addition makes the configuration unstable, unstable vertices are toppled until a stable configuration is obtained. The abelian sandpile Markov chain eventually reaches the set of recurrent configurations, and after entering this set it stays there forever. Furthermore, the stationary distribution of the chain is the uniform distribution on the recurrent configurations. 

The abelian sandpile model has been investigated extensively from many different perspectives. These include the study of height probabilities and average heights \cite{P94,JPR06,PPR11}, the dynamics of the abelian sandpile Markov chain and properties such as its mixing behavior \cite{JLP19,HJL19}, as well as critical exponents. Examples of the latter include exponents describing the probability that a vertex topples after a particle is added at the origin \cite{JRS15}, the total number of topplings, and the radius of an avalanche \cite{BHJ17,H20}.
In addition to the $d$-dimensional lattice $\mathbb{Z}^d$ \cite{LP09,FLP10}, the model has also been studied on many other classes of graphs. These include ladder graphs, formed as the Cartesian product of a finite connected graph with an interval \cite{JL07}, the Bethe lattice \cite{DM90,MRS02}, and fractal graphs such as the Sierpiński gasket graph \cite{CK20,KSW24,HKS25} and the Vicsek fractal graph \cite{HKS25_vic}.

In this paper, we study the abelian sandpile model on the two-dimensional comb lattice $\mathcal{C}_2$, the graph obtained from $\mathbb{Z}^2$ by deleting all horizontal edges except those lying on the $x$-axis. The comb lattice is of particular interest because it exhibits several unusual properties: for instance, two independent random walks on $\mathcal{C}_2$ intersect almost surely only finitely many times \cite{KP04}, and the Einstein relation between the fractal, spectral, and walk dimensions fails to hold \cite{B06}. Random walks on the comb lattice, as well as the corresponding local limit theorems, have been studied extensively \cite{G86,WH86,CD92,CCFR09}.

The first part of the paper concerns the infinite volume limit of the stationary measures and the size and shape of avalanches under these measures in finite volume. On transient graphs, the expected number of topplings at a site $y$ during the stabilization of $\eta+\delta_x$ is equal to the Green function $G(x,y)$, and is therefore finite. Here $\eta$ is a stable sandpile sampled from the uniform volume limit measure $\nu$ and $\eta +\delta_x$ denotes the addition of one particle at vertex $x$ to the configuration $\eta$. Moreover, it is known on $\Z^d$ for $d\geq 3$ that for $\nu$-almost every stable sandpile configuration $\eta$ and every $x\in V$, the configuration $\eta+\delta_x$ can be stabilized using only finitely many topplings \cite{JR08, J18}. 
For recurrent graphs, by contrast, the question of whether $\eta+\delta_x$ can be stabilized in finitely many topplings is still not fully understood, as the expected number of topplings is infinite. In \cite{B26}, it has been shown that avalanches on $\Z^2$ stay finite almost surely. The infinite Vicsek fractal graph provides an example of a recurrent graph on which \(\eta+\delta_o\) stabilizes with probability \(3/4\) and explodes with probability \(1/4\) \cite{HKS25_vic}. 

All objects and concepts appearing below will be defined precisely in Section~\ref{sec:preliminaries}. On the two dimensional comb lattice $\mathcal{C}_2$, let $\nu_n$ denote the stationary distribution of the sandpile Markov chain on the box of side length $2n+1$ centered at the origin, with the exterior boundary identified as a sink (that is, particles that leave the box during stabilization are lost forever). We refer to this choice as wired boundary conditions. In the infinite volume limit, we show that, since $\mathcal{C}_2$ is a recurrent tree, the measures $\nu_n$ converge to a limit measure supported on the saturated configuration. We then study the shape of avalanches induced by adding  to a stable sandpile $\eta$ sampled from $\nu_n$ a particle at the origin, and establish a discrepancy between the vertical and horizontal spread of avalanches in finite boxes  of side length $2n+1$. The set of vertices that topple during the stabilization of $\eta+\delta_{o}$ (the sandpile obtained from adding to $\eta$ a particle at the origin $o$) is called the avalanche induced by $\eta+\delta_o$ and is denoted by $\text{Av}(\eta,o)$. For a set $A\subseteq\mathcal{C}_2$ we define its vertical and horizontal spread as
$$\text{Hor}(A)=\sup\{|x|\;:\; (x,y)\in A\},\hspace{1cm}\text{Ver}(A)=\sup\{|y|\;:\;(x,y)\in A\}.$$

\begin{theorem}\label{thm:main}
    Consider the two dimensional comb lattice $\mathcal{C}_2$, and let $\nu_n$ be the stationary distribution of the abelian sandpile Markov chain on the finite box of side length $2n+1$ centered at the origin, equipped with wired boundary conditions.  Let $\eta$ be a stable sandpile on the finite box of side length $2n+1$ sampled from $\nu_n$. Then it holds:
    \begin{enumerate}[label=(\alph*)]
    \setlength\itemsep{0em}
        \item\label{thm:maina} (Infinite volume limit) The measures $(\nu_n)_{n\in\N}$ converge weakly to the Dirac measure $\delta_{\eta_\text{max}}$ supported on the saturated configuration on $\mathcal{C}_2$, defined as $\eta_{\text{max}}(v)=\deg(v)-1$ for all $v\in\mathcal{C}_2$.
        \item\label{thm:mainb} (Vertical spread) The avalanche induced by $\eta+\delta_o$ reaches the boundary along a
        vertical tooth with high probability, that is
        $$\lim_{n\rightarrow\infty}\nu_n(\text{Ver}(\text{Av}(\eta,o))\geq n)=1.$$
        \item\label{thm:mainc} (Horizontal spread) The avalanche induced by $\eta+\delta_o$ has a horizontal spread of order $\sqrt{n}$ up to sub-polynomial corrections, that is, for every $\delta>0$ we have
        $$\lim_{n\rightarrow\infty}\nu_n(\text{Hor}\big(\text{Av}(\eta,o))\in[n^{1/2-\delta},n^{1/2+\delta}]\big)=1.$$
    \end{enumerate}
\end{theorem}
Regarding the infinite volume limit part of the above result, it is not surprising that it concentrates on the saturated (maximal) configuration. Indeed, in Section~\ref{sec:inf-vol} we prove the corresponding statement for all infinite recurrent trees. The reason for this phenomenon is that, on any recurrent tree, the infinite volume limit of the wired uniform spanning forests is supported on a single tree, namely the entire underlying graph. The analogous conclusion for the infinite volume limit of abelian sandpiles then follows from the burning bijection.
Theorem~\ref{thm:main}~\ref{thm:mainb} shows that, in finite boxes, avalanches already reach the boundary with high probability along the vertical teeth. To prove this, we analyze the wired uniform spanning forest on finite subgraphs of $\mathcal{C}_2 $ and show that, under the stationary distribution, there is with high probability a path of vertices with maximal height reaching the boundary along a vertical tooth. A refinement of the analysis of the path of vertices with maximal height then yields Theorem~\ref{thm:main}~\ref{thm:mainc}.

Interestingly, the horizontal spread of avalanches is of the order of the square root of the vertical spread. This same scaling relation also appears in the limit shapes of three aggregation models on the comb lattice \cite{IDLA-comb-2012}: the divisible sandpile, internal diffusion-limited aggregation, and rotor-router aggregation.
This observation naturally motivates our second main result: we prove that the limit shape obtained by placing $n$ particles at the origin $o=(0,0)$ of $\mathcal{C}_2$ and stabilizing according to the abelian sandpile toppling rule,
coincides with the limit shape of the three aforementioned aggregation models when $n\to\infty$.
The divisible sandpile \cite{LP09} is the continuous analogue of the abelian sandpile model, in which particles may be split and a continuous mass is redistributed over the state space. In internal diffusion-limited aggregation (IDLA) \cite{DF90}, an aggregate is formed by independent random walks, each of which moves until it first reaches a previously unoccupied vertex. Rotor-router aggregation is a deterministic, derandomized, counterpart of IDLA \cite{PDDK96}.
The appearance of the same limit shape across these different models is referred to as \emph{limit shape universality}. This phenomenon has been established only for a few state spaces, such as the Sierpiński gasket \cite{CK20,FHKS24,Div-sand-gasket-2019,IDLA-gasket-2020}. In fact, on $\mathbb{Z}^2$, it is believed that the abelian sandpile limit shape is not a Euclidean ball, in contrast to the limit shape of the other three models \cite{Rotor-Zd-Levine-Peres-2008,LP09,Levine-Peres-2017-survey}. However, on $\Z^d$ the existence of the scaling limit has been shown using techniques from PDE theory in \cite{PS13}, and later these findings have been refined in \cite{LPS16,B21}

We add the comb lattice $\mathcal{C}_2$ to the class of graphs exhibiting limit shape universality, by determining the limit shape of the single source abelian sandpile. In particular, we prove that, on the comb lattice $\mathcal{C}_2$, the abelian sandpile model has the same limit shape as the corresponding aggregation models \cite{IDLA-comb-2012}.
For $n>0$, define the set
\begin{align}\tag{LS}\label{eq:Bm}
\mathcal B_n= \big\{(x,y)\in\mathcal{C}_2\;:\;|x|\le R_n,
    \ |y|\le H_n(x)
  \big\},
\end{align}
where
$R_n=\left(\frac{9n}{4}\right)^{1/3}$ and $H_n(x)=\frac13\bigl(R_n-|x|\bigr)_+^2.$

\begin{theorem}\label{thm:main2}(Limit shape of the abelian sandpile on $\mathcal{C}_2$)
 For $n\in\mathbb{N}$,  let $\mathcal{S}_n $ denote the set of vertices in $\mathcal{C}_2$ that topple during the stabilization of the sandpile configuration $n\delta_o$  of $n$ particles placed at $o=(0,0)$. Then, for every \(0<\varepsilon<1\), there exists \(N(\varepsilon)\in\mathbb{N}\) such that, for all \(n\ge N(\varepsilon)\), it holds
\[
\mathcal{B}_{(1-\varepsilon)n}\subseteq \mathcal{S}_n\subseteq \mathcal{B}_{(1+\varepsilon)n}.
\]
\end{theorem}
For the proof of Theorem~\ref{thm:main2}, we establish the inner and outer bounds separately. The inner bound is the simpler part: using the least action principle for divisible sandpiles, we compare the abelian sandpile cluster with the divisible sandpile cluster. A key point in this comparison is that we define the divisible sandpile model so that vertices on the backbone have capacity $3$. The outer bound requires a different argument. We construct the odometer corresponding to the stabilization of $n_k\delta_o$ along a suitable subsequence $(n_k)_{k\in\mathbb{N}}$, and then use the monotonicity of the limit shape with respect to the number of particles initially placed at the origin to obtain the desired result.

\textbf{Outline.} 
In Section~\ref{sec:preliminaries}, we introduce the abelian sandpile model, the comb lattice, and the additional concepts needed for the proofs of our results. In Section~\ref{sec:inf-vol}, we prove Part~\ref{thm:maina} of Theorem~\ref{thm:main} for all recurrent trees. Section~\ref{sec:avalanche} is devoted to the proof of Parts~\ref{thm:mainb} and~\ref{thm:mainc} of Theorem~\ref{thm:main}, concerning the vertical and horizontal spread of avalanches in stationarity. Finally, in Section~\ref{sec:single-source}, we establish the single source limit shape stated in Theorem~\ref{thm:main2}. We conclude with a open questions section.

\section{Preliminaries}
\label{sec:preliminaries}

\textbf{The abelian sandpile model (ASM).}
A comprehensive overview of the abelian sandpile model can be found in \cite{J18}. 
Consider a finite, undirected, connected graph $G=(V \cup \{s\},E)$ where $s$ is a distinguished vertex that acts as a sink. A sandpile configuration on $G$ is a function $\sigma:V \rightarrow \N$,  with $\sigma(v)$ representing the number of particles at the vertex $v$. 
A vertex $v \in V$ is called stable, if $\sigma(v)<\deg(v)$, where $\deg(v)$ denotes the degree of $v$, that is, the number of its neighboring vertices; otherwise, $v$ is called unstable. A sandpile configuration is stable if all vertices in $V$ are stable. A toppling at a vertex consists of sending one particle along each incident edge to its neighbors. Such a toppling is called legal if it is performed only at an unstable vertex. Any particles that reach the sink are absorbed and removed from the system. Since the sink can store particles indefinitely, its particle count is not recorded.
The dynamics of the ASM can be described using the graph Laplacian, a square matrix $\Delta\in {\rm Mat}_{V\cup \{s\}}(\Z)$ whose entries are defined as
\[\Delta_{vw}=
\begin{cases}
    -e_{vw}, &  v \neq w, \, v \sim w,\\
    \deg(v), & v = w,\\
    0, & \text{otherwise,}
\end{cases}\]
where $v \sim w$ indicates that $v$ and $w$ are neighbors in $G$, i.e.~that they are connected by an edge, and $e_{vw}$ is the number of edges between $v$ and $w$. The reduced graph Laplacian $\tilde{\Delta} $ of $G$ is obtained from the graph Laplacian $\Delta$ by removing the row and column corresponding to the sink vertex. We define the toppling operator, which gives the resulting sandpile configuration after toppling a sandpile $\sigma$ at vertex $v \in V$, by
\[(T_v\sigma)(u)=
\begin{cases}
    \sigma(u)-\deg(u), & u=v,\\
    \sigma(u) + e_{vu}, & u \sim v,\\
    \sigma(u), & \text{otherwise.}
\end{cases}\]
Viewing $\sigma$ as a vector indexed by $V$, whose entries record the number of particles at the corresponding vertices, the toppling of a vertex $v$ can equivalently be written in terms of the reduced graph Laplacian as $T_v\sigma=\sigma-\widetilde{\Delta}\delta_v$,
where $\delta_v$ denotes the standard basis vector with entry $1$ at $v$ and entry $0$ at all other vertices. This formulation immediately shows the abelian property of the model: toppling operators commute, that is, $T_vT_u\sigma=T_uT_v\sigma$.
The stabilization $\sigma^\circ$ of a sandpile configuration $\sigma$ is the unique stable configuration obtained after carrying out all possible legal topplings. More precisely, there exists a sequence of legal topplings such that $\sigma^\circ = T_{v_n}\cdots T_{v_1}\sigma$, where the set of toppled vertices $\{v_1,\dots,v_n\}$ is uniquely determined, although the order in which these topplings are performed is not, as a consequence of the abelian property. For  details, see \cite{J18}.

\textbf{Abelian sandpile Markov chain.} 
The dynamics of the abelian sandpile model defines a Markov chain on the set of stable sandpile configurations on $G$. Let $(X_n)_{n\in\mathbb{N}}$ be an i.i.d. sequence of random variables, uniformly distributed on $V$, and let $\sigma_0$ be an arbitrary stable configuration. The abelian sandpile Markov chain $(\sigma_n)_{n\in\mathbb{N}}$ is then defined recursively as follows: given $\sigma_n$, add one particle at the randomly chosen vertex $X_n$ and stabilize the resulting configuration, namely $\sigma_{n+1}=(\sigma_n+\delta_{X_n})^\circ$.
This Markov chain eventually reaches the set of recurrent configurations, which constitutes a unique communicating class. Its stationary distribution is the uniform measure on the set of recurrent configurations; see, for instance, \cite{J18}.

\textbf{Burning bijection.} Dhar \cite{D90}  introduced the burning algorithm as an efficient procedure for determining whether a given stable sandpile is recurrent. The algorithm leads to the burning bijection, due to Majumdar and Dhar \cite{MD92}, establishing a one-to-one correspondence between recurrent sandpile configurations and spanning trees of the underlying graph. In this paper, we use the setup from \cite{LP14}. For each vertex $x\in V$, fix a total ordering $<_x$ on the set $E_x$ of edges incident to $x$. Let $T$ be a spanning tree of $G$. For each vertex $x\neq s$, let $e_T(x)$ denote the first edge on the unique path in $T$ from $x$ to the sink, and let $l_T(x)$ denote the length of this path, that is, the number of edges it contains. The burning bijection assigns to the spanning tree $T$ a recurrent sandpile configuration $\sigma_T: V \rightarrow \N$ defined by $\sigma_T(x)= \deg(x) - 1 - a_T(x) - b_T(x)$
where
\begin{align*}
    a_T(x)&= \vert \{(x,y) \in E_x : l_T(y)<l_T(x)-1\}\vert\\
    b_T(x)&= \vert \{(x,y) \in E_x : l_T(y)=l_T(x)-1 \text{ and } (x,y)<_x e_T(x)\}\vert.
\end{align*}
Under this characterization, a vertex $x$ attains maximal height $\deg(x)-1$ if and only if $a_T(x)=b_T(x)=0$.
The condition $a_T(x)=0$ is equivalent to saying that there is no neighbor $y$ of $x$ whose distance to the sink along the tree $T$ is at least two smaller than that of $x$. The condition $b_T(x)=0$ means that, among those neighbors whose path to the sink is shorter by exactly one, the edge $e_T(x)$ is the smallest with respect to the ordering $<_x$.

\begin{figure}
\begin{minipage}[b]{0.45\textwidth}
\centering
    \begin{tikzpicture}[scale=0.6]
        \foreach \x in {-3,...,2} {
            \draw (\x,0) -- (\x+1,0);}
        
        \foreach \x in {-3,...,3} {
            \foreach \y in {-3,...,2} {
                \draw (\x,\y) -- (\x,\y+1);}
        }
                
        \foreach \x in {-3,...,3} {
            \fill (\x,0) circle (2pt);
            
            \foreach \y in {-3,...,3} {
                \fill (\x,\y) circle (2pt);}
        }
        \node[below left] at (0,0) {$o$};
    \end{tikzpicture}
\end{minipage}
\begin{minipage}[b]{0.45\textwidth}
\centering
    \begin{tikzpicture}[scale=0.6]
        \coordinate (s) at (2.5,4.5);
              
        \foreach \x in {-3,...,3} {
            \draw[lightgray] (\x,3) -- (s);
            \draw[lightgray] (\x,-3) -- (s);}
        \draw[lightgray] (-3,0) -- (s);
        \draw[lightgray] (3,0) -- (s);

        \foreach \x in {-3,...,2} {
            \draw (\x,0) -- (\x+1,0);}
        
        \foreach \x in {-3,...,3} {
            \foreach \y in {-3,...,2} {
                \draw (\x,\y) -- (\x,\y+1);}
        }
                
        \foreach \x in {-3,...,3} {
            \fill (\x,0) circle (2pt);
            
            \foreach \y in {-3,...,3} {
                \fill (\x,\y) circle (2pt);}
        }

        \fill (s) circle (3pt);
        \node[above] at (s) {$s$};
        \node[below left] at (0,0) {$o$};

    \end{tikzpicture}
    \end{minipage}
\caption{A finite portion of $\mathcal{C}_2$ on the left and the comb lattice $\mathcal{C}_2$ with sink $s$ on the right.}
\label{fig:comb-lattice-finite}
    \end{figure}
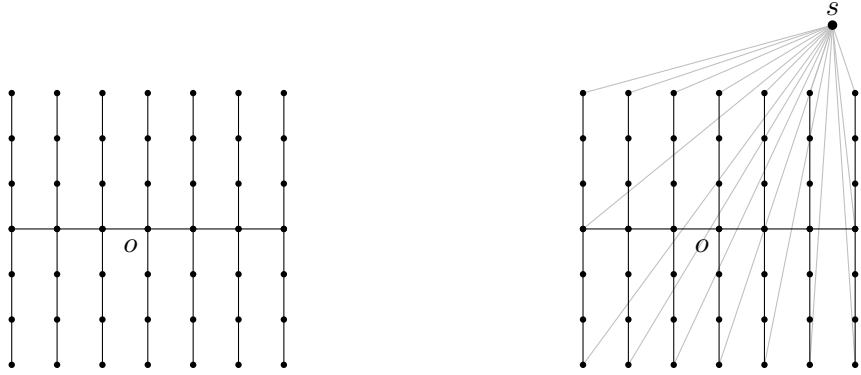

\textbf{Comb lattice.}
The two-dimensional comb lattice $\mathcal{C}_2$ is obtained from $\mathbb{Z}^2$ by deleting all horizontal edges except those lying on the $x$-axis. A finite portion of $\mathcal{C}_2$ is shown in Figure~\ref{fig:comb-lattice-finite}. Equivalently, one may construct $\mathcal{C}_2$ by starting with the integer line $\mathbb{Z}$, called the \textit{backbone}, and attaching to each of its vertices a vertical copy of $\mathbb{Z}$, consisting of two \textit{teeth}. 
Using the standard embedding of $\mathcal{C}_2$ in $\mathbb{Z}^2$, we label vertices by Cartesian coordinates $z=(x,y)\in\mathbb{Z}^2$. With this notation, $\mathcal{C}_2=(V,E)$ has vertex set $V=\mathbb{Z}^2$ and edge set
\[
E=\{((x,n),(x,m)):\ |n-m|=1\}\cup\{((x,0),(w,0)):\ |x-w|=1\}.
\]
In this paper, we study the abelian sandpile model on finite subgraphs of the comb lattice $\mathcal{C}_2$. We add a sink vertex $s$, through which particles are removed from the system. More precisely, we restrict $\mathcal{C}_2$ to the box of size $(2n+1)\times(2n+1)$ centered at the origin $o=(0,0)$. We then connect the boundary vertices $(x,n)$ and $(x,-n)$, for $x\in\{-n,\ldots,n\}$, as well as the endpoints $(-n,0)$ and $(n,0)$ of the backbone, to the sink vertex as in Figure~\ref{fig:comb-lattice-finite}. We denote by $\nu_n$ the stationary distribution of the sandpile Markov chain on this finite box.

\textbf{The infinite volume limit.}
We now consider the extension of stationary distributions of the sandpile Markov chain to infinite graphs. As before, $G=(V,E)$ is a locally finite, connected, infinite graph. An exhaustion of $G$ is a sequence of finite vertex sets 
$V_1\subset V_2\subset \cdots \subset V$
such that $\bigcup_{i=1}^\infty V_i=V$. For each $n\in\mathbb{N}$, let $G_n=(V_n\cup\{s_n\},E_n)$ be the finite graph obtained by identifying all vertices in $V\setminus V_n$ with a single sink vertex. Denote by $\mu_n$ the stationary distribution of the sandpile Markov chain on $G_n$. 
The infinite volume limit $\mu$ on $G$, when it exists, is defined as the weak limit of the measures $(\mu_n)_{n\in\mathbb{N}}$. More precisely, we say that $(\mu_n)_{n\in\mathbb{N}}$ converges weakly if, for every finite set $A\subseteq V$ and every function $\zeta:A\to\mathbb{Z}$, the limit 
$\lim_{n\to\infty}\mu_n(\eta|_A=\zeta)$
exists. In this case, we define
\[
\mu(\eta|_A=\zeta)
:=
\lim_{n\to\infty}\mu_n(\eta|_A=\zeta).
\]

\textbf{Divisible sandpile.} The divisible sandpile model is a continuous counterpart of the abelian sandpile model, in which mass may be divided during redistribution. To define it precisely, fix a height function $h:\mathcal{C}_2\to\mathbb{R}$, representing the mass capacity at each vertex. A sand distribution is a finitely supported function $\mu:\mathcal{C}_2\to\mathbb{R}_{\geq 0}$. For a vertex $x\in\mathcal{C}_2$, the toppling operator $D_x$ is 
\[
D_x\mu=\mu-\frac{\max\{\mu(x)-h(x),0\}}{\deg(x)}\Delta\delta_x .
\]
Thus, when $x$ is toppled, the excess mass above its capacity $h(x)$ is distributed equally among its neighbors, thus the term \emph{divisble sandpile}.
Let $\mu_0$ be an initial sand distribution, and fix a sequence of vertices $(x_k)_{k\in\mathbb{N}}$ in which every vertex occurs infinitely often. We define a sequence of sand distributions $(\mu_k)_{k\in\mathbb{N}_0}$ inductively by
$\mu_k=D_{x_k}\mu_{k-1}$, $k\geq 1$. We also define the corresponding odometer sequence $(u_k)_{k\in\mathbb{N}}$ by
\[
u_k(y)=\frac{1}{{\deg(y)}}\sum_{\substack{j\leq k:\\ x_j=y}}(\mu_{j-1}(y)-\mu_{j}(y)),
\qquad y\in\mathcal{C}_2 .
\]
The odometer describes the total amount of mass emitted from each vertex during the toppling procedure. A fundamental property of the limiting odometer $u_\infty=\lim_{n\to\infty}u_n$ is that it is the pointwise minimal non-negative function satisfying
$\mu_0-\Delta u_\infty\leq h $.
This characterization is known as the least action principle \cite{Rotor-Zd-Levine-Peres-2008}.
Finally, we note that it is known from \cite{IDLA-comb-2012} that, for the divisible sandpile on the comb lattice started from the sand distribution $n\delta_o$, the cluster of toppled vertices is given by \eqref{eq:Bm}, up to fluctuations of constant order.

\section{Infinite volume limit}\label{sec:inf-vol}

In this section, we prove Theorem~\ref{thm:main}~\ref{thm:maina}. More precisely, we show that the infinite volume limit exists and is the Dirac measure supported on the maximal (saturated) sandpile configuration. The argument applies not only to the comb lattice $\mathcal{C}_2$, but to the broader class of recurrent trees. On recurrent graphs, the wired and free uniform spanning trees have the same infinite volume limit \cite[Proposition~5.6]{BLPS01}. We use this result to show that, on recurrent trees, any recurrent sandpile configuration eventually has maximal height $\deg-1$ when restricted to a fixed finite subset set of vertices.

\begin{proposition}\label{prop:rec-tree}
    Let $\mathcal{T}$ be a recurrent tree, and fix a vertex $r\in\mathcal{T}$, which we call the root. Let $\mu_n$ denote the stationary distribution of the abelian sandpile Markov chain on the ball $B(r,n)$ of radius $n$ centered at $r$, with wired boundary conditions. Then the sequence $(\mu_n)_{n\in\mathbb{N}}$ converges weakly to the Dirac measure $\delta_{\eta_{\mathrm{max}}}$ supported on the saturated sandpile configuration on $\mathcal{T}$.
\end{proposition}
\begin{proof}
    Fix $m\in\mathbb{N}$ and for each $n\in\mathbb{N}$, let $\eta_n$ be a sandpile configuration sampled according to $\mu_n$. By the burning bijection, recurrent sandpile configurations on $B(r,n)$ are in one-to-one correspondence with wired spanning trees of $B(r,n)$. Since $\mathcal{T}$ is recurrent, the wired uniform spanning trees and the free uniform spanning trees have the same infinite volume limit. Moreover, because $\mathcal{T}$ is itself a tree, the infinite volume limit of the free uniform spanning trees is simply $\mathcal{T}$. Hence the uniform volume limit of the wired uniform spanning trees is given by $\mathcal{T}$ itself.

    If $T(\eta_n)$ denotes the spanning tree obtained from $\eta_n$ via the burning bijection, then, with high probability, every vertex $x\in B(r,m)$ is connected to the unique path $\gamma$ from the root $r$ to the boundary of $B(r,n)$. Here, “with high probability” means with probability tending to $1$ as $n\to\infty$. It follows that, for every $x\in B(r,m-1)$, there is a unique neighbor $y\sim x$ such that $l_{T(\eta_n)}(y)<l_{T(\eta_n)}(x)$, while all other neighbors of $x$ are descendants of $x$ in $T(\eta_n)$. Consequently, $a_{T(\eta_n)}(x)=b_{T(\eta_n)}(x)=0$. By the burning bijection, this implies that $\eta_n(x)=\deg_{\mathcal{T}}(x)-1$. Since this holds for every vertex in $B(r,m-1)$, we obtain
 \[
 \lim_{n\to\infty}\Prob(\eta_n\vert_{B(r,m-1)}=\eta_{\text{max}}\vert_{B(r,m-1)})=1.
 \]
Thus the sequence $(\mu_n)_{n\in\mathbb{N}}$ has infinite volume limit $\delta_{\eta_{\mathrm{max}}}$.
\end{proof}

\begin{proof}[Proof of Theorem~\ref{thm:main}~\ref{thm:maina}]
   Since the comb $\mathcal{C}_2$ is a recurrent tree, the claim follows immediately from Proposition~\ref{prop:rec-tree}.
\end{proof}

\section{Structure of avalanches in finite volumes}\label{sec:avalanche}

Theorem~\ref{thm:main}~\ref{thm:maina} establishes that if a sandpile is sampled from the infinite volume limit and a particle is added at the origin, then every vertex topples infinitely many times, that is, avalanches in the infinite volume limit are infinite. This result, however, does not describe the finite volume structure of avalanches before taking the limit. In this section we study avalanches on finite subgraphs of the comb lattice, and we show that the maximal vertical spread of an avalanche has asymptotic behavior different from that of its maximal horizontal spread.

We first prove that, with high probability, there is a path of maximal height vertices starting at the origin and reaching the sink whose horizontal displacement is of order $\sqrt{n}$. We do this by studying the exit location from a finite box of a random walk started at the origin $o=(0,0)$ in the comb lattice, and then constructing appropriate blocking events which ensure that every vertex along the loop-erasure of this random walk path has maximal height.
We then show that the horizontal spread of an avalanche is smaller than $n^{1/2+\delta}$ for every $\delta>0$, by locating the first missing horizontal edge to the right and to the left of the origin in the wired uniform spanning tree. Throughout this section, unless stated otherwise, we assume $0<\delta<1/2$, since Part~\ref{thm:mainc} of Theorem~\ref{thm:main} is trivial for $\delta\geq 1/2$.

\subsection{Full height path}\label{subsec:full-height}

For a sandpile configuration $\eta$, we say that a vertex $x$ has full height or maximal height if $\eta(x)=\deg(x)-1$, and we call a path of consecutive vertices \emph{full height path} if it consists of only full height vertices.
For $n\in \N$, consider the finite box of side length $2n+1$, centered at the origin $o=(0,0)$ in the comb lattice $\mathcal{C}_2$, with its exterior boundary identified with the sink. We prove that, if a recurrent sandpile $\eta_n$ is sampled from the stationary distribution $\nu_n$ of the sandpile Markov chain on this box, then with probability tending to $1$ as $n\to\infty$, there exists a path of maximal height vertices starting at $o$ and reaching the sink through one of the teeth, whose horizontal displacement along the backbone is at least $n^{1/2-\delta}$ for every $\delta>0$.
Indeed, if $\eta_n$ contains a path of maximal height vertices beginning at the origin and ending at the sink, then adding one particle at $o$ forces every vertex on this path to topple. Therefore, the existence of such a full height path proves Theorem~\ref{thm:main}~\ref{thm:mainb} and yields the lower bound in Theorem~\ref{thm:main}~\ref{thm:mainc}. In what follows, whenever we refer to uniform spanning trees in a finite box, we mean wired uniform spanning trees.

We characterize the existence of a full height path from the origin to the sink, whose backbone part has length $k\in\Z$, in terms of the following four sufficient conditions on the spanning tree $T(\eta_n)$ obtained from applying the inverse burning bijection to the sandpile $\eta_n$ sampled from $\nu_n$. See Figure \ref{fig:comb-lattice_events} for an illustration of these events.

\begin{enumerate}[leftmargin=1.8cm, labelwidth=1.57cm]
    \item[$OSP(k)$]\label{osp} Origin-sink path of length $k$: there exists a direct path in the spanning tree from the origin to the sink vertex such that $(k,0)$ is the vertex on this path with maximal horizontal displacement.
    \item[$TH(k)$]\label{thk}   Tooth-hole condition: on each tooth attached to the backbone vertices of the path in $OSP(k)$, other than the one that itself belongs to the path in $OSP(k)$, there is a missing edge in the spanning tree, that is not adjacent to the backbone.
    \item[$PO(k)$]\label{rok} Path obstruction: both neighbors of $(k,0)$ along the backbone are connected to $(k,0)$ in the spanning tree.
    \item[$RHO(k)$]\label{lok} Remaining half obstruction: both neighbors of the origin along the backbone are connected to the origin in the spanning tree.
\end{enumerate}

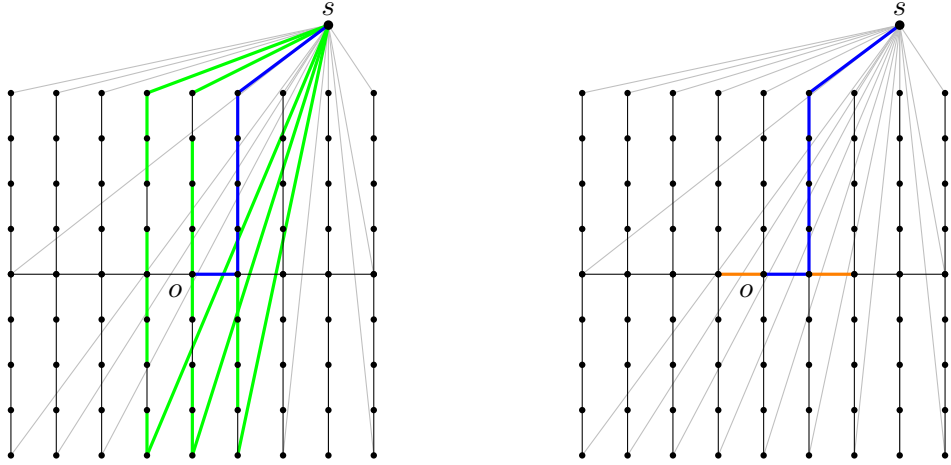
\begin{figure}
\centering
\begin{minipage}{0.45\textwidth}
\centering
    \begin{tikzpicture}[scale=0.6]
        \coordinate (s) at (3,5.5);
        
        \foreach \x in {-4,...,4} {
            \draw[lightgray] (\x,4) -- (s);
            \draw[lightgray] (\x,-4) -- (s);
        }
        \draw[lightgray] (-4,0) -- (s);
        \draw[lightgray] (4,0) -- (s);

        \draw[green, very thick] (-1,-4)--(s);
        \draw[green, very thick] (-1,4)--(s);
        \draw[green, very thick] (0,-4)--(s);
        \draw[green, very thick] (0,4)--(s);
        \draw[green, very thick] (1,-4)--(s);

        \foreach \x in {-4,...,3} {
            \draw (\x,0) -- (\x+1,0);
        }
        \foreach \x in {-4,...,4} {
            \foreach \y in {-4,...,3} {
                \draw (\x,\y) -- (\x,\y+1);
            }
        }

        \draw[green, very thick] (0,0)--(0,1)--(0,2)--(0,3);

        \draw[green, very thick] (0,0)--(0,-1);
        \draw[green, very thick] (0,-2)--(0,-3)--(0,-4);

        \draw[green, very thick] (1,0)--(1,-1)--(1,-2)--(1,-3);

        \draw[green, very thick] (-1,0)--(-1,1);
        \draw[green, very thick] (-1,2)--(-1,3)--(-1,4);

        \draw[green, very thick] (-1,0)--(-1,-1)--(-1,-2);
        \draw[green, very thick] (-1,-3)--(-1,-4);

        \draw[blue, very thick] (0,0)--(1,0)--(1,1)--(1,2)--(1,3)--(1,4)--(s);

        \foreach \x in {-4,...,4} {
            \fill (\x,0) circle (2pt);
        
            \foreach \y in {-4,...,4} {
                \fill (\x,\y) circle (2pt);
            }
        }
        \fill (s) circle (3pt);
        \node[above] at (s) {$s$};
        \node[below left] at (0,0) {$o$};
    \end{tikzpicture}
\end{minipage}
\begin{minipage}{0.45\textwidth}
\centering
    \begin{tikzpicture}[scale=0.6]
        \coordinate (s) at (3,5.5);
        \foreach \x in {-4,...,4} {
            \draw[lightgray] (\x,4) -- (s);
            \draw[lightgray] (\x,-4) -- (s);
        }
        \draw[lightgray] (-4,0) -- (s);
        \draw[lightgray] (4,0) -- (s);

        \foreach \x in {-4,...,3} {
            \draw (\x,0) -- (\x+1,0);
        }
        \foreach \x in {-4,...,4} {
            \foreach \y in {-4,...,3} {
                \draw (\x,\y) -- (\x,\y+1);
            }
        }
        
        \draw[blue, very thick] (0,0)--(1,0)--(1,1)--(1,2)--(1,3)--(1,4)--(s);
        \draw[orange, very thick] (1,0)--(2,0);
        \draw[orange, very thick] (-1,0)--(0,0);
        
        \foreach \x in {-4,...,4} {
            \fill (\x,0) circle (2pt);
        
            \foreach \y in {-4,...,4} {
                \fill (\x,\y) circle (2pt);
            }
        }
        \fill (s) circle (3pt);
        \node[above] at (s) {$s$};
        \node[below left] at (0,0) {$o$};
    \end{tikzpicture}
\end{minipage}
\caption{The comb $\mathcal{C}_2$ restricted to a $9\times9$ box with sink vertex $s$. The blue path represents an origin-sink path of length $1$ on the backbone $OSP(1)$. On the left, the green paths satisfy the tooth-hole condition $TH(1)$. On the right, the orange paths correspond to the obstructions $PO(1)$ and $RHO(1)$.}
\label{fig:comb-lattice_events}
\end{figure}

First, observe that if $k_1\neq k_2$, then the events $OSP(k_1)$ and $OSP(k_2)$ are disjoint, since any spanning tree contains a unique path from the origin to the sink. Furthermore, if we condition on $OSP(k)$, the path connecting the origin to the sink splits the spanning tree into independent parts, from which it follows that the remaining three events are conditionally independent. 
By the burning bijection, if for some $k\in\mathbb{N}$ all four events occur simultaneously, then there is a full height path from the origin to the sink whose horizontal spread is \(|k|\). Hence our aim is to estimate
$\Prob(OSP(k))$
together with the conditional probabilities
\[
\Prob(TH(k)\mid OSP(k)),\qquad
\Prob(PO(k)\mid OSP(k)),\qquad
\Prob(RHO(k)\mid OSP(k)),
\]
for \(k\in\mathbb{Z}\).
We start with the event $OSP(k)$ for $k\in\Z$.
By Wilson’s algorithm, the event that the uniform spanning tree contains a path from $o$ to the sink such that $(k,0)$ is the vertex of maximal horizontal displacement on this tree path is equivalent to the event that a random walk started at $o$ exits the $(2n+1)\times(2n+1)$ box through one of the teeth attached at $(k,0)$. The probability of this event can be estimated using a one-dimensional killed random walk. Indeed, whenever the walker is at a backbone vertex and moves into one of the attached teeth, it leaves the box before returning to the backbone with probability $\frac{1}{n}$. Let $(\xi_i)_{i\in\mathbb{N}}$ be a sequence of independent identically distributed random variables satisfying $\mathbb{P}(\xi_1=1)=\mathbb{P}(\xi_1=-1)=\frac{1}{2}$,
and define $S_t=\sum_{i=1}^t \xi_i$. We also let $E\sim \operatorname{Geom}\big(\frac{1}{n+1}\big)$ a geometric random variable independent of $(\xi_i)_{i\in\mathbb{N}}$.

\begin{lemma}\label{lem:direct_path}
Let $k\in\mathbb{Z}$, and consider the comb lattice $\mathcal{C}_2$ restricted to the centered box of size $(2n+1)\times(2n+1)$. Let $T$ be a wired uniform spanning tree on this finite graph. Then the probability that $T$ contains a direct path from the origin $o=(0,0)$ to the sink such that $(k,0)$ is the vertex of maximal horizontal displacement along the backbone on this path is given by
$$
\mathbb{P}\bigl(S_E=k \,\big|\, |S_t|\leq n \text{ for all } t\leq E\bigr),
$$
where $S_E=\sum_{i=1}^E \xi_i$.
\end{lemma}

\begin{proof}
    By Wilson’s algorithm, the path from the origin to the sink in the uniform spanning tree is distributed as the loop-erasure of a random walk started at $o$ and stopped when it exits the box of size $(2n+1)\times(2n+1)$. At each backbone vertex, this walk moves into one of the two adjacent teeth with probability $1/2$, and once it has entered a tooth, gambler’s ruin estimates show that it reaches the sink before returning to the backbone with probability $1/n$.
    Thus, the probability of connecting to the sink at horizontal displacement $k$ can be represented in terms of a one-dimensional random walk that is killed at $k$, without ever leaving the interval $[-n,n]$. Indeed, if the walk exits this interval along the backbone, then the origin is connected to the sink through the backbone instead. At each backbone vertex, the probability of reaching the sink before making a horizontal step to the left or to the right can be computed by summing over the number of excursions into the attached teeth. More precisely, the walker may enter a tooth $l$ times and reach the sink on its $l$-th such attempt, for some $l\in\mathbb{N}$. Hence the corresponding killing probability is given by
    $$\frac{1}{2n}\sum_{l=0}^\infty \left(\frac{n-1}{2n}\right)^l=\frac{1}{2n}\frac{2n}{n+1}=\frac{1}{n+1}.$$
\end{proof}
It remains to bound the probability that the random walk is killed at a distance of order $\sqrt{n}$.

\begin{lemma}\label{lem:estimate-killed-at-sqrt}
For a simple random walk $(S_t)_{t\in\mathbb{N}_0}$ on $\mathbb{Z}$, killed at an independent geometric random time $E\sim \operatorname{Geom}\left(\frac{1}{n+1}\right)$, we have
$$\lim_{n\to\infty}\mathbb{P}\Big(n^{1/2-\delta}\leq |S_E|\leq n^{1/2+\delta}\,\Big|\,
|S_t|\leq n \text{ for all } t\leq E\Big)=1.
$$
\end{lemma}
\begin{proof}
    Let $\varepsilon>0$, and fix a constant $C>0$, to be chosen later. With high probability, the geometric clock $E$ rings before time $Cn$, and over a time interval of order $n$, a simple random walk on $\Z$ typically reaches distances of order $\sqrt{n}$. Standard large deviation estimates imply $\lim_{n\to\infty}\Prob\bigl(\exists\, t\leq Cn:\ |S_t|>n\bigr)=0$.
     Thus, it is enough to estimate the probabilities  $\Prob\big(|S_E|\geq n^{1/2-\delta}\big)$ and $\Prob\big(|S_E|\leq n^{1/2+\delta}\big)$. Since $E$ is independent of the random walk, we get
    $$\Prob\big(|S_E|\geq n^{1/2-\delta}\big)\geq\sum_{k=C^{-1}n}^{Cn}\Prob(E=k)\Prob\big(|S_k|\geq n^{1/2-\delta}\big).$$
    Let us choose $C>0$ such that $\Prob\big(C^{-1}n \leq E \leq Cn\big) \geq 1-\varepsilon$. Notice also that, by the central limit theorem, there exists $N_0 \in \mathbb{N}$ such that, for every $n \geq N_0$ and every $k$ with $C^{-1}n \leq k \leq Cn$, we have $\Prob\big(\frac{|S_k|}{k} \geq n^{1/2-\delta}\big) \geq 1-\varepsilon$. Combining these two observations, we obtain 
    $$\sum_{k=C^{-1}n}^{Cn}\Prob(E=k)\Prob\big(|S_k|\geq n^{1/2-\delta}\big)\geq (1-\varepsilon)\Prob\big(C^{-1}n\leq E\leq Cn\big)\geq (1-\varepsilon)^2.$$
    Since $\varepsilon>0$ was arbitrary, we obtain $\lim_{n\to\infty}\Prob(|S_E|\geq n^{1/2-\delta})=1$. 
    To bound $\Prob\big(|S_E|\leq n^{1/2+\delta}\big)$, again by  the central limit theorem, for all sufficiently large $n$,
  $\Prob\big(\frac{|S_k|}{k}\leq n^{1/2+\delta}\big)>1-\varepsilon$ for $k$ satisfying $C^{-1}n\leq k\leq Cn$. Arguing as before, we obtain $\lim_{n\to\infty}\Prob\big(|S_E|\leq n^{1/2+\delta}\big)=1$. Combining this with the corresponding lower bound, proves the claim.
\end{proof}

Using the construction of paths in a uniform spanning tree via killed random walks, together with the estimates from Lemma \ref{lem:estimate-killed-at-sqrt}, we can now prove the existence of a full height path with horizontal displacement at least $n^{1/2-\delta}$.

\begin{proposition}\label{prop:full-height-path}
For a sandpile $\eta_n$ sampled according to the stationary measure $\nu_n$ of the abelian sandpile Markov chain on the box of side length $2n+1$ centered at the origin $o=(0,0)$ of the comb $\mathcal{C}_2$, with probability tending to $1$ as $n\to\infty$, there exists a simple path $\gamma:[0,t]\to \mathcal{C}_2$, where $t\in\mathbb{N}$, which starts at $o=(0,0)$ and reaches the boundary of the box through one of the teeth, such that
$$\eta_n\big(\gamma(i)\big)=\deg(\gamma(i))-1\text{ for all }0\leq i\leq t,\hspace{0.5cm}\text{ and }\hspace{0.5cm}\text{Hor}\big(\{\gamma(i)\;|\;0\leq i\leq t\}\big)\geq n^{1/2-\delta}.$$
\end{proposition}

\begin{proof}
    By the burning bijection, it suffices to show that as $n\to\infty$
    \[
    \sum_{\substack{k\in\Z:\\ n^{1/2-\delta}\leq|k|\leq n^{1/2+\delta}}}\Prob\big(TH(k)\;|\;OSP(k)\big)\Prob\big(PO(k)\;|\;OSP(k)\big)\Prob\big(RHO(k)\;|\;OSP(k)\big)\Prob\big(OSP(k)\big)\to 1.
    \]
    Fix an arbitrary $\varepsilon>0$ and consider $k\in\mathbb{Z}$ satisfying $n^{1/2-\delta}\leq |k|\leq n^{1/2+\delta}$. We first estimate the conditional probability $\Prob\big(TH(k)\mid OSP(k)\big)$.
    Conditioned on the event $OSP(k)$, the teeth attached to the origin-sink path are sampled independently. Since each tooth is one-dimensional, the probability that the edge adjacent to the backbone is absent in $T(\eta_n)$ equals $1/(n+1)$. Therefore,
    \[
    \Prob\big(TH(k)\;|\;OSP(k)\big)=\Big(1-\frac{1}{n+1}\Big)^{2|k|+1}\geq \Big(1-\frac{1}{n+1}\Big)^{2n^{1/2+\delta}+1}.
    \]
    The right-hand side tends to $1$ as $n\to\infty$ provided that $\delta<1/2$. Hence, we may choose $N_1\in\mathbb{N}$ such that, for all $n\geq N_1$,  $\Prob\big(TH(k)\mid OSP(k)\big)\geq 1-\varepsilon $.
    Next we consider the probability of the path obstruction event $\Prob\big(PO(k)\;|\;OSP(k)\big)$.
    We again use Wilson's algorithm to characterize this event in terms of a random walk started from the backbone neighbor of $(k,0)$ that does not belong to the origin-sink path in the event $OSP(k)$. For the event $PO(k)$ not to occur, the random walk must leave the box of size $(2n+1)\times(2n+1)$ before hitting $(k,0)$. Since $|k|\leq n^{1/2+\delta}$, the remaining distance along the backbone to the boundary of the box is of order $n$, and the same is true for the distances along the teeth.   Thus, the probability of this event goes to $0$ as $n\to \infty$. Therefore, we can choose $N_2\in\mathbb{N}$ such that, for every $n\geq N_2$, $\Prob\big(PO(k)\mid OSP(k)\big)\geq 1-\varepsilon$. Analogously, there exists $N_3\in\mathbb{N}$ such that, for all $n\geq N_3$, it holds that $\Prob\big(RHO(k)\mid OSP(k)\big)\geq 1-\varepsilon$.
    All together, we obtain
    \begin{align*}
        \sum_{k\in\Z:n^{1/2-\delta}\leq|k|\leq n^{1/2+\delta}}\Prob\big(TH(k)&\;|\;OSP(k)\big)\Prob\big(PO(k)\;|\;OSP(k)\big)\Prob\big(RHO(k)\;|\;OSP(k)\big)\Prob\big(OSP(k)\big)\\
        &\geq (1-\varepsilon)^3\sum_{k\in\Z:n^{1/2-\delta}\leq|k|\leq n^{1/2+\delta}}\Prob\big(OSP(k)\big).
    \end{align*}
    The sum on the right-hand side above converges to $1$ by Lemma \ref{lem:estimate-killed-at-sqrt} and the discussion above Lemma \ref{lem:estimate-killed-at-sqrt}, so the claim follows since $\varepsilon>0$ was arbitrary.
\end{proof}
From this, we obtain Theorem~\ref{thm:main}~\ref{thm:mainb}  and the lower bound in Theorem~\ref{thm:main}~\ref{thm:mainc}.

\begin{proof}[Proof of Theorem~\ref{thm:main}~\ref{thm:mainb}  and the lower bound in Theorem~\ref{thm:main}~\ref{thm:mainc}]
Observe that if $\eta_n$ contains a path of maximal height vertices starting at the origin and reaching the sink through one of the teeth, then adding a particle to $\eta_n$ at the origin necessarily produces an avalanche with vertical spread equal to $n$. By Proposition \ref{prop:full-height-path}, such a path with horizontal displacement at least $n^{1/2-\delta}$ exists with probability tending to $1$, from which the claim follows.
\end{proof}

\subsection{Horizontal spread of avalanches}\label{sec:hor-spread}

Here we prove the upper bound in Theorem~\ref{thm:main}~\ref{thm:mainc}.  The results of Subsection \ref{subsec:full-height} already imply that, for every $\delta>0$, the horizontal spread of an avalanche is, with high probability, larger than $n^{1/2-\delta}$. It remains to establish the matching upper bound. For this purpose, it is enough to find vertices on the backbone, both to the left and to the right of the origin, that are within distance $n^{1/2+\delta}$ and are not of maximal height $3$. In general, the presence of vertices that do not have maximal height does not by itself control the size of an avalanche. However, on the comb lattice it does suffice, as shown by the following lemma.

\begin{lemma}\label{lem:not-full-height-aval}
  Take a stable sandpile configuration $\eta:\mathcal{C}_2\to\mathbb{Z}$, and suppose that there exists $m\in\mathbb{Z}$ such that $\eta((m,0))<3$. Then the vertex $(m,0)$ does not topple during the stabilization of $\eta+\delta_o$.
\end{lemma}

\begin{proof}
    Assume, without loss of generality, that $m>0$. Also, let us assume that $m$ is the smallest such value, that is for all $i<m$ it holds that $\eta((i,0))=3$. For $k\in\mathbb{N}$, define the $k$-th wave as the set of vertices that topple when, during the stabilization of $\eta+\delta_o$, each unstable vertex is allowed to topple at most $k$ times. Thus, in the first wave, the origin $o$ is toppled once, and then every other unstable vertex is toppled once, while $o$ is not allowed to topple a second time. Since the comb lattice is a tree, the vertex $(m,0)$ cannot topple during the first wave. Consequently, after the first wave, the vertex $(m-1,0)$ contains at most $2$ particles. It follows that $(m-1,0)$ cannot topple during the second wave, and hence $(m,0)$ cannot topple during the second wave either. If $m=1$, the avalanche stops at this point. If $m\geq 2$, then after the second wave the vertex $(m-2,0)$ contains at most $2$ particles, since $(m-1,0)$ did not topple. Continuing this argument inductively, we conclude that $(m,0)$ never topples during the stabilization. This 
    proves the claim
\end{proof}

To find a vertex in the sandpile configuration that is not of maximal height, it is enough to identify a missing backbone edge in the corresponding uniform spanning tree, as shown by the following lemma.

\begin{lemma}\label{lem:height-edge-missing}
For $n\in\mathbb{N}$, let $\eta_n$ be sampled from the stationary measure $\nu_n$ of the abelian sandpile Markov chain on the box of side length $2n+1$ centered at the origin $o=(0,0)$ in the comb lattice $\mathcal{C}_2$. Fix $0<\delta<1/2$, and let $(m,0)$ and $(m+1,0)$ be adjacent backbone vertices with $0\leq m\leq n^{1/2+\delta}$.
Let $M_m$ denote the event that the edge $\{(m,0),(m+1,0)\}$ is the first missing backbone edge to the right of the origin in the spanning tree $T(\eta_n)$. Then it holds that
    $$\lim_{n\to\infty}\mathbb{P}\big(\eta_n\big((m,0)\big)<3\text{ or }\eta_n((m+1,0))<3 \;\vert\; M_m\big)=1.$$
\end{lemma}

\begin{proof}
    By the burning bijection, if $\eta_n$ is a recurrent sandpile such that $\{(m,0),(m+1,0)\}\notin T(\eta_n)$, then $\eta_n((m,0))=\eta_n((m+1,0))=3$, and so the lengths of their respective paths to the sink in $T(\eta_n)$ differ by at most one, that is,  $\vert l_{T(\eta_n)}((m,0))-l_{T(\eta_n)}((m+1,0))\vert\leq 1$. Thus
    \[
    \Prob\big(\eta_n((m,0))=3\text{ and }\eta_n((m+1,0))=3\;\vert\; M_m\big)\leq \Prob(|l_{T(\eta_n)}((m,0))-l_{T(\eta_n)}((m+1,0))|\leq 1\;\vert\; M_m)\,.
    \]
    Conditionally on $M_m$, the spanning tree $T(\eta_n)$ splits into two independent components. In the component containing $(m+1,0)$, the quantity $l_{T(\eta_n)}((m+1,0))$ is the length of the loop-erased random walk from $(m+1,0)$ to the sink, as described in Subsection \ref{subsec:full-height}. Since all teeth have the same length, this length is determined only by the horizontal distance traveled by the loop-erased random walk. We denote this horizontal distance by $X_{m+1}$. Then it holds
    \begin{align*}
        \Prob\big(|l_{T(\eta_n)}((m,0))-&l_{T(\eta_n)}((m+1,0))|\leq 1\;\vert\; M_m\big)=\Prob(\vert l_{T(\eta_n)}((m,0))-X_{m+1}-n\vert\leq 1\;\vert\;M_m)\\&=\sum_{k\in\N}\Prob(l_{T(\eta_n)}((m,0))=n+k\;\vert\;M_m)\Prob(X_{m+1}\in\{k-1,k,k+1\})
        \\&\leq 3\max_{k\in\N}\Prob(X_{m+1}=k) \sum_{k\in\N}\Prob(l_{T(\eta_n)}((m,0))=n+k\;\vert\;M_m)\\&=3\max_{k\in\N}\Prob(X_{m+1}=k).
    \end{align*}
    From the representation of $X_{m+1}$ via a one dimensional killed random walk as in Subsection \ref{subsec:full-height} and one dimensional local CLTs, it follows that $\lim_{n\to\infty}3\max_{k\in\N}\Prob(X_{m+1}=k)=0$. Notice that we use for this limit that $m\leq n^{1/2+\delta}$. Thus the claim follows.
\end{proof}
From the above proof we also obtain
\begin{equation}\label{eq:pb-missing-edge}
    \lim_{n\to\infty}\max_{m\in[0,n^{1/2+\delta}]}\Prob\Big(\forall m\in\{0,...,n^{1/2+\delta}\}:\eta_n((m,0))=3\;\big\vert\;M_m\Big)=0.
\end{equation}

It remains to show that, with high probability, there is a missing edge in the uniform spanning tree at a sufficiently small distance to the right of the origin. This is established in the following lemma.

\begin{lemma}\label{lem:missing-edge}
Let $0<\delta<1/2$, and let $T_n$ be the uniform spanning tree on the $(2n+1)\times(2n+1)$ box of the comb lattice with wired boundary conditions. Then
\[
\lim_{n\to\infty}\Prob\Big(\exists\, m\in\{0,\ldots,n^{1/2+\delta}\}:\{(m,0),(m+1,0)\}\notin T_n\Big)=1.
\]
Thus, with high probability, there is a missing horizontal edge in $T_n$ within distance $n^{1/2+\delta}$ of $o=(0,0)$.
\end{lemma}
\begin{proof}
    Let us choose the vertex $v=(\lfloor n^{1/2+\delta}\rfloor,0)$. Observe that if the unique path connecting $v$ and $o$ in the uniform spanning tree $T_n$ passes through the sink, then at least one edge along the backbone path between $o$ and $v$ must be missing in $T_n$. By Wilson's algorithm, the probability that $o$ is connected to $v$ in the uniform spanning tree is bounded above by the probability that two random walks, one started at $o$ and the other at $v$, intersect before exiting the box of side length $2n+1$ centered at $o$. By Lemma \ref{lem:estimate-killed-at-sqrt}, this probability goes to $0$ as $n\to\infty$. Thus, with probability tending to $1$, there must be a missing edge between $o$ and $v$, which proves the claim. 
\end{proof}

We now combine the preceding observations to conclude that, with high probability, there exists a vertex at horizontal distance at most $n^{1/2+\delta}$ that does not have maximal height.

\begin{proposition}\label{prop:not-full-height}
    For $n\in\mathbb{N}$, let $\eta_n$ be sampled from the stationary measure $\nu_n$ of the abelian sandpile Markov chain on the box of side length $2n+1$ centered at $o=(0,0)$ in the comb lattice $\mathcal{C}_2$. Fix $0<\delta<1/2$. Then it holds that
    $$\lim_{n\to\infty}\Prob\Big(\exists m\in\{0,...,n^{1/2+\delta}\}:\eta_n((m,0))<3\Big)=1.$$
    That is, there is a vertex of non-maximal height whose horizontal distance from $o=(0,0)$ is at most $n^{1/2+\delta}$.
\end{proposition}
\begin{proof}
    Let $M$ denote the event that some horizontal edge at distance at most $n^{1/2+\delta}$ is absent from the spanning tree $T(\eta_n)$, and define the events $M_m$ as in Lemma \ref{lem:height-edge-missing}. Then it holds
    \begin{align*}
    \Prob\big(\forall m\in\{0,...,n^{1/2+\delta}\}:\eta_n((m,0))=3\big)&=\Prob\big(\forall m\in\{0,...,n^{1/2+\delta}\}:\eta_n((m,0))=3,M\big)\\&+\Prob\big(\forall m\in\{0,...,n^{1/2+\delta}\}:\eta_n((m,0))=3,M^c\big).
    \end{align*}
    By Lemma \ref{lem:missing-edge}, the second term on the right-hand side above goes to $0$ as $n\to\infty$, so it remains to consider the first term. We have
    \begin{align*}
        \Prob\big(\forall m\in\{0,...,n^{1/2+\delta}\}:& \eta_n((m,0))=3,  M\big)=\sum_{i=0}^m\Prob\big(\forall m\in\{0,...,n^{1/2+\delta}\}:\eta_n((m,0))=3,M_i\big)\\
        &=\sum_{i=0}^m\Prob(\forall m\in\{0,...,n^{1/2+\delta}\}:\eta_n((m,0))=3\;\vert\;M_i)\Prob(M_i).
    \end{align*}
    Using \eqref{eq:pb-missing-edge}, we get the claim. 
\end{proof}
We can now prove the upper bound in Theorem~\ref{thm:main}~\ref{thm:mainc}.
\begin{proof}[Proof of the upper bound in Theorem~\ref{thm:main}~\ref{thm:mainc}]
By Proposition \ref{prop:not-full-height}, with probability tending to $1$, there exists a vertex to the right of $o=(0,0)$, at distance at most $n^{1/2+\delta}$, whose height is not maximal $3$. By symmetry, the same statement holds for a vertex to the left of the origin. Hence, the probability that such non-maximal height vertices exist on both sides of the origin, each within distance at most $n^{1/2+\delta}$, also tends to $1$ as $n\to\infty$, since this event is the intersection of the two preceding events. The desired upper bound then follows from Lemma \ref{lem:not-full-height-aval}.
\end{proof}

\begin{figure}
    \centering
    \begin{minipage}{0.32\textwidth}
        \centering
        \includegraphics[height=8.6cm]{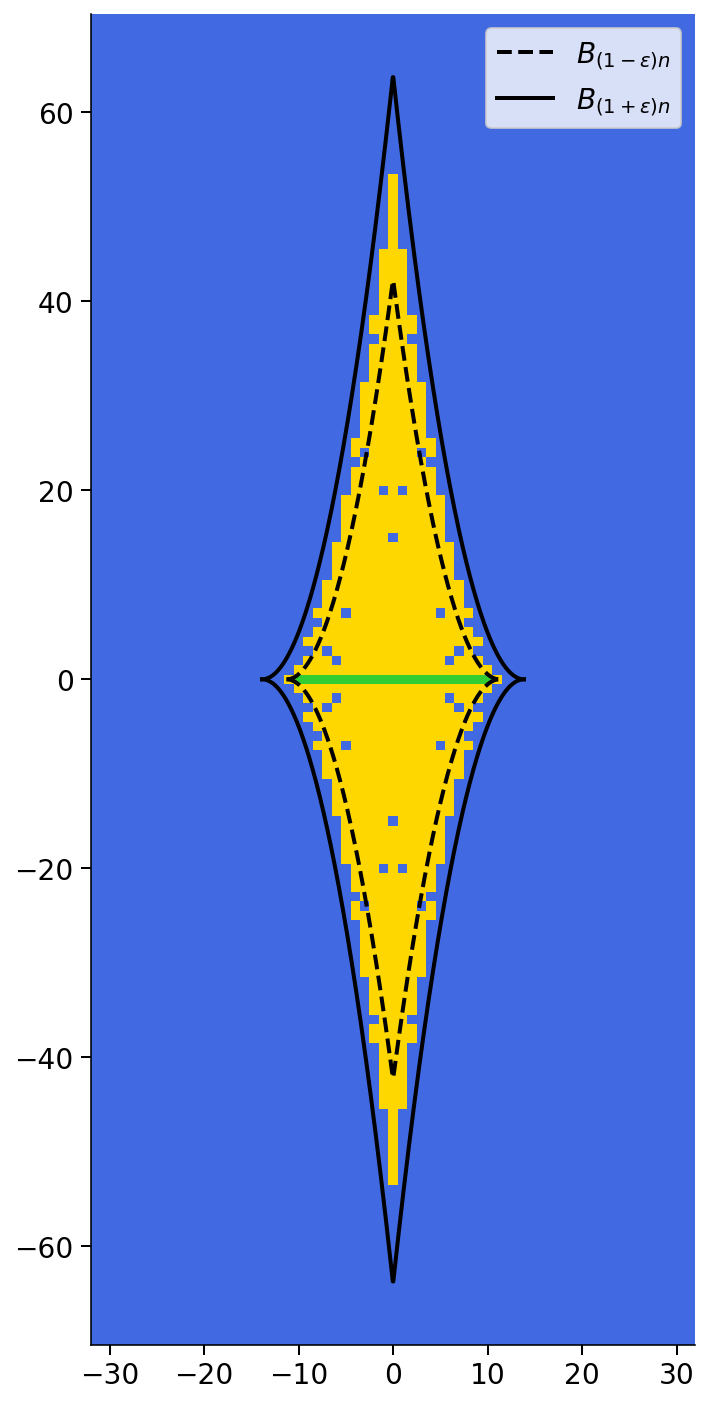}
    \end{minipage}
    \hfill
    \begin{minipage}{0.32\textwidth}
        \centering
        \includegraphics[height=8.6cm]{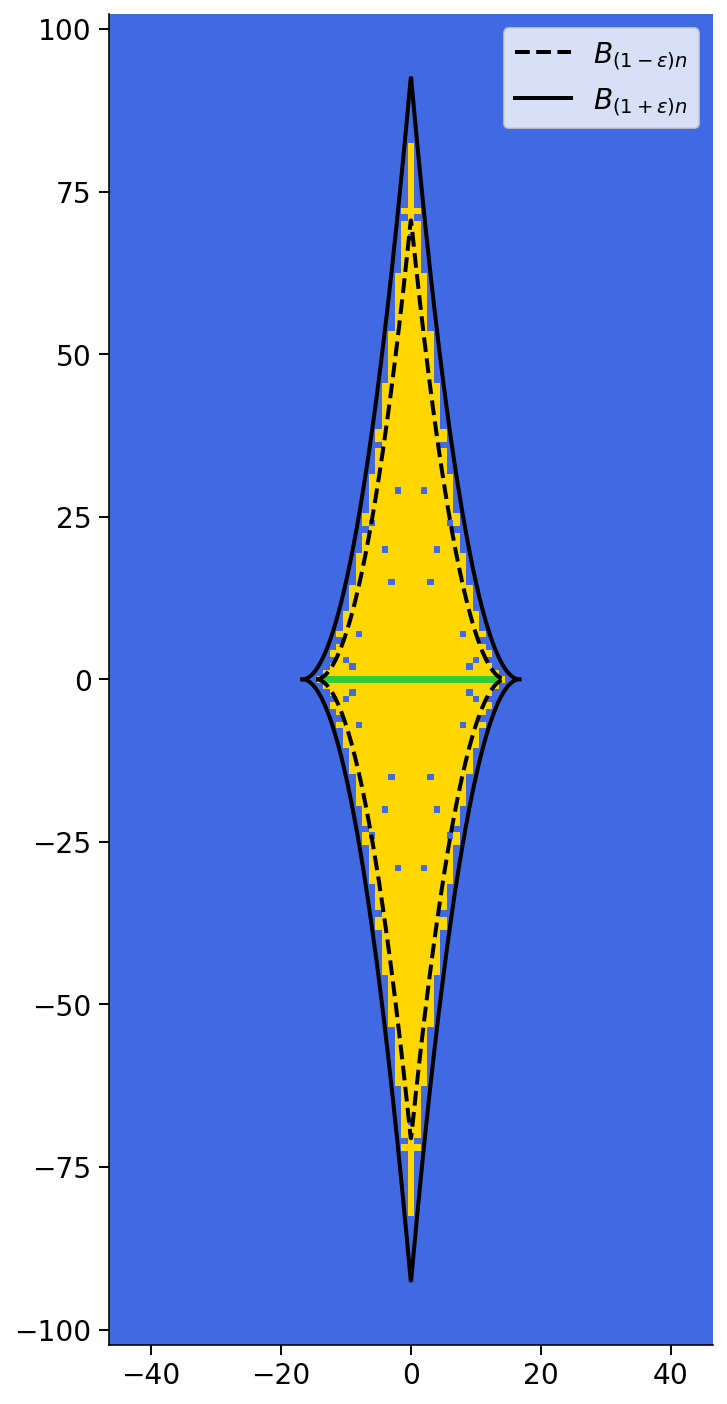}
    \end{minipage}
    \hfill
    \begin{minipage}{0.32\textwidth}
        \centering
        \includegraphics[height=8.7cm]{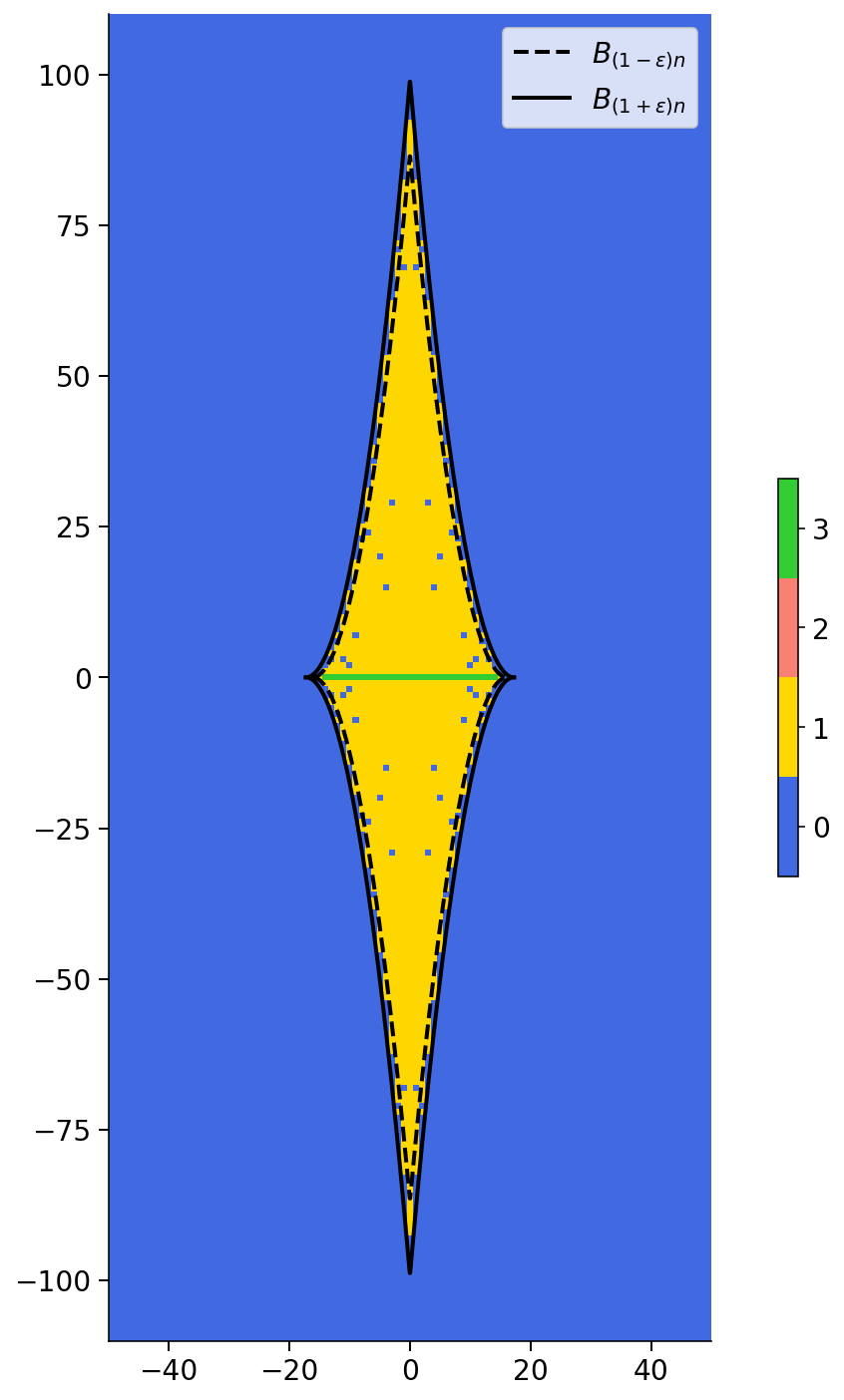}
    \end{minipage}
    \caption{The limit shape of the abelian sandpile on the comb, obtained by placing $n$ particles at $o=(0,0)$ and stabilizing. The dashed lines indicate the boundary of $\mathcal{B}_{(1-\varepsilon)n}$, while the solid lines indicate the boundary of $\mathcal{B}_{(1+\varepsilon)n}$. The parameter values are $n=909$ and $\varepsilon=0.3$ on the left, $n=1721$ and $\varepsilon=0.2$ in the middle, and $n=2073$ and $\varepsilon=0.1$ on the right. Vertex colors correspond to the number of particles present at each site, as indicated by the colorbar on the right. }
    \label{fig:limit_shape_with_ball}
\end{figure}

\section{Single-source limit shape}\label{sec:single-source}
In this section we prove Theorem \ref{thm:main2} on the limit shape of the single source abelian sandpile model on the comb lattice $\mathcal{C}_2$. In Figure \ref{fig:limit_shape_with_ball}, we display the limit shapes for different values of $n$, together with the boundaries of the sets $\mathcal{B}_{(1-\varepsilon)n}$ and $\mathcal{B}_{(1+\varepsilon)n}$. The figure displays the set of vertices that receive at least one particle during stabilization, which differs from the limit shape of toppled vertices $\mathcal{S}_n$ by the outermost layer.

\subsection{Inner bound}

Consider the divisible sandpile on $\mathcal{C}_2$, with capacity $3$ assigned to backbone vertices and capacity $1$ assigned to tooth vertices. If we show that this divisible sandpile has limit shape $\mathcal{B}_m$ as in \eqref{eq:Bm}, then the least action principle yields an inner bound for the abelian sandpile limit shape, because the divisible sandpile odometer is the minimal non-negative function that stabilizes $n$ particles. In fact, for our purposes it is enough to construct a subsolution to the divisible sandpile odometer.

Let $n\in\N$ be the total mass placed at the origin $o=(0,0)$ of $\mathcal{C}_2$ and define for $x\in\Z$
$$R_n=\Big\lfloor\Big(\frac{9n}{4}\Big)^{1/3}\Big\rfloor,\hspace{1cm}z_x=\max(0,R_n-|x|),\hspace{1cm}
K_x^{(n)}=\max\Big(0,\frac{1}{3}z_x^2-2\Big).$$
For the divisible sandpile with the above capacities (3 on the backbone vertices  and 1 on the teeth vertices), define a subsolution of the odometer function via
\begin{align}\label{eq:lower-bound-function}L^{(n)}(x,y)=\frac{1}{2}\max(0,K_x^{(n)}-|y|)^2+\frac{1}{2}\max(0,K_x^{(n)}-|y|).\end{align}
By definition it holds that $L^{(n)}\geq 0$. In order to show that this is indeed a valid subsolution for the divisible sandpile odometer, we need to show the following three properties of $\Delta L^{(n)}$.
\begin{enumerate}[label=(\alph*)]
    \item\label{eq:lower-bound-a} For all $(x,0)\in\mathcal{C}_2$ with $0<|x|< R_n-2$ it holds $-\Delta L^{(n)}(x,0)\geq 3$.
    \item\label{eq:lower-bound-b} For all $(x,y)\in\mathcal{C}_2$ with $|x|\leq R_n-2$ and $1\leq |y|\leq K_x^{(n)}$  it holds $-\Delta L^{(n)}(x,y)\geq 1$.
    \item\label{eq:lower-bound-c} $-\Delta L^{(n)}(o)\geq 3-n$.
\end{enumerate}
We verify these three properties one by one.

\ref{eq:lower-bound-a} For $x\neq 0$ with $z_x>3$ it holds
$$-\Delta L^{(n)}(x,0) =\frac{1}{2}\Big(\big(K_{x-1}^{(n)}\big)^2+\big(K_{x+1}^{(n)}\big)^2-2\big(K_{x}^{(n)}\big)^2 +K_{x-1}^{(n)}+K_{x+1}^{(n)}-2K_{x}^{(n)}\Big)+2\big(L^{(n)}(x,1)-L^{(n)}(x,0)\big).$$
We analyze the terms on the right-hand side separately. Since $\big(K_x^{(n)}\big)^2=\frac{1}{9}z_x^4-\frac{4}{3}z_x^2+4$, we get
$$(K_{x-1}^{(n)})^2+(K_{x+1}^{(n)})^2-2(K_{x}^{(n)})^2=\frac{12}{9}z_x^2+\frac{2}{9}-\frac{8}{3}=\frac{12}{9}z_x^2-\frac{22}{9}.$$
Similarly we obtain $K_{x-1}^{(n)}+K_{x+1}^{(n)}-2K_{x}^{(n)}=\frac{2}{3}$. Moreover
$$2L^{(n)}(x,1)-2L^{(n)}(x,0)=(K^{(n)}_x-1)^2+(K_x^{(n)}-1)-(K_x^{(n)})^2-K_x^{(n)}=-2K_x^{(n)},$$
which yields
$$-\Delta L^{(n)}(x,0)=\Big(\frac{12}{18}z_x^2-\frac{8}{9}\Big)-\Big(\frac{2}{3}z_x^2-4\Big)=\frac{28}{9}>3.$$
We still need to consider the Laplacian at the horizontal boundary. For $z_x=3$, we have $K_x^{(n)}=1$ and inserting this into the Laplacian, yields
$-\Delta L^{(n)}(x,0)=\frac{100}{18}+\frac{30}{18}-2-2=\frac{29}{9}>3.$ 

\ref{eq:lower-bound-b} Let us first consider $|y|\leq K_x^{(n)}-1$, then $L$ along the tooth is a a quadratic polynomial, whose discrete Laplacian evaluates to $1$. On the boundary, where $|y|+1>K_x^{(n)}$ we have
$$-\Delta L^{(n)}(x,y)=\frac{1}{2}\left((K_x^{(n)}-|y|+1)^2-2(K_x^{(n)}-|y|)^2\right)+\frac{1}{2}\left(|y|+1-K_x^{(n)}\right).$$
Since
$$(K_x^{(n)}-|y|+1)^2-2(K_x^{(n)}-|y|)^2=1+2(K_x^{(n)}-|y|)-(K_x^{(n)}-|y|)^2,$$
we get
$$-\Delta L^{(n)}(x,y)=\frac{1}{2}\left(-(K_x^{(n)}-|y|)^2+2+(K_x^{(n)}-|y|)\right)=1+\frac{1}{2}\left(K_x^{(n)}-|y|\right)\left(1-(K_x^{(n)}-|y|)\right),$$
 and since $K_x^{(n)}-|y|\leq 1$ we obtain $-\Delta L^{(n)}(x,y)\geq 1.$

 \ref{eq:lower-bound-c} 
By explicitly evaluating the function at $(0,1)$, $(0,-1)$, $(1,0)$, and $(-1,0)$, and then substituting these values into the Laplacian, we obtain
 $$-\Delta L^{(n)}(o)=-\frac{4}{9}R_n^3+\frac{14}{9}R_n+\frac{28}{9},$$
which together with 
 $R_n^{3}\leq \frac{9n}{4}$ yields  $-\Delta L^{(n)}(o)\geq 3-n$.

Using the function $L^{(n)}$, we can now establish a lower bound for the divisible sandpile limit shape where the capacity for the backbone vertices is 3 and for all the other vertices is 1.

\begin{lemma}\label{lem:lower-bound-divisible}
Consider the divisible sandpile on $\mathcal{C}_2$, where backbone vertices have capacity $3$ and tooth vertices have capacity $1$. Let $n \in \mathbb{N}$, and let $w_n$ denote the divisible sandpile odometer obtained by placing an initial mass of $n$ at $o=(0,0)$ and then stabilizing the configuration. Then it holds 
\[
\mathcal{L}_n:=\Big\{(x,y)\in\mathcal{C}_2\;:\; |x|< R_n-2\text{ and }|y|< \frac{1}{3}\big(R_n-|x|\big)^2-2\Big\}\subseteq \text{supp}(w_n)\,.
\]
\end{lemma}
\begin{proof}
    For $L^{(n)}$ defined in \eqref{eq:lower-bound-function} it holds
    $$-\Delta (w_n-L^{(n)})\leq 0 \quad \text{on } \mathcal{L}_n$$
    and on the external boundary $\partial \mathcal{L}_n=\big\{z\notin \mathcal{L}_n\; : \;\exists z'\in\mathcal{L}_n:z\sim z'\big\}$  of $\mathcal{L}_n$ we have
    $$(w_n-L^{(n)})(x,y)=w_n(x,y)\geq 0.$$
    It thus follows that $w_n-L^{(n)}$ must be non-negative everywhere in $\mathcal{L}_n$ by the minimum principle, from which our claim follows.
\end{proof}

With this lemma established, we can now prove the inner bound for the abelian sandpile limit shape.

\begin{proposition}\label{prop:inner-bound}
 For $n\in\mathbb{N}$,  let $\mathcal{S}_n $ denote the set of vertices in $\mathcal{C}_2$ that topple during the stabilization of the sandpile configuration $n\delta_o$  of $n$ particles placed at $o=(0,0)$. Then, for every \(\varepsilon>0\), there exists \(N(\varepsilon)\in\mathbb{N}\) such that, for all \(n\ge N(\varepsilon)\), it holds
   $ \mathcal{B}_{(1-\varepsilon)n}\subseteq \mathcal{S}_n$.
\end{proposition}
\begin{proof}
Stabilize the sandpile $n\delta_o$ using the abelian sandpile toppling procedure and the divisible sandpile toppling procedure, where backbone vertices are assigned capacity $3$ and tooth vertices capacity $1$. Denote the corresponding odometer functions by $v_n$ for the abelian sandpile and $w_n$ for the divisible sandpile. By the least action principle, we have $w_n\leq v_n$. Lemma \ref{lem:lower-bound-divisible}  together with $\mathcal{S}_n=\{(x,y)\in\mathcal{C}_2:\ v_n(x,y)>0\}$ implies that
    $$\mathcal{L}_n:=\Big\{(x,y)\in\mathcal{C}_2\;:\; |x|< R_n-2\text{ and }|y|< \frac{1}{3}(R_n-|x|)^2-2\Big\}\subseteq\mathcal{S}_n.$$
    Choosing $N(\varepsilon)\in\N$ such that for all $n\geq N(\varepsilon)$ it holds $\left((1-\varepsilon)\frac{9n}{4}\right)^{1/3}\leq R_n-3$, then for $x\in\Z$ with $|x|\leq \left((1-\varepsilon)\frac{9n}{4}\right)^{1/3}$ we have
    \[
    \frac{1}{3}\Big(\Big((1-\varepsilon)\frac{9n}{4}\Big)^{1/3}-|x|\Big)^2\leq \frac{1}{3}(R_n-|x|)^2-2
    \]
    which implies $\mathcal{B}_{(1-\varepsilon)n}\subseteq \mathcal{L}_n\subseteq \mathcal{S}_n.$ 
\end{proof}

\subsection{Outer bound}

In this subsection, we construct an integer valued function that will serve as an upper bound for the odometer function of the abelian sandpile on $\mathcal{C}_2$. To motivate the construction, we first examine the final stable configuration obtained by stabilizing an initial point mass, as illustrated in Figure \ref{fig:limit_shape_with_ball}. For a particular sequence of initial masses, this final stable configuration has a well-controlled structure: it is identically $3$ along the backbone, while on each tooth it is identically $1$ except for exactly one hole. This suggests constructing an appropriate upper bound by finding a function whose Laplacian is constant equal to $3$ on the backbone and constant equal to $1$ on the teeth, and then perturbing it by a wedge function along the teeth in order to prescribe the locations of the holes. To simplify the construction, rather than beginning with a specific point mass and then building a corresponding upper bound, we instead fix the horizontal spread of the final limit shape along the $x$-axis and construct a function with the desired Laplacian values and support.  A similar approach has been used for the limit shape of the rotor-router aggregation on the comb lattice in \cite{Huss-Sava-2012-RRComb}.
To this end, fix some $R \in \mathbb{N}$. Let $(f_R(x))_{x \in \mathbb{N}_0}$ be defined as the solution to the following recursion. Set $f_R(R)=1$ and $f_R(x)=0$ for all $x>R$. For all remaining values of $x$, define
\begin{align}\label{eq:odometer-backbone}\tag{Rec-BB}
    f_R(x)   = 2f_R(x+1)-f_R(x+2) +2\Big\lfloor \frac{1+\sqrt{1+8f_R(x+1)}}{2}\Big\rfloor +1.
\end{align}
This function is uniquely determined, since the recursion can be solved backwards from $x=R$ using the prescribed boundary values. 
The solution to \eqref{eq:odometer-backbone} will ultimately provide the backbone values of our upper bound. The recursion is chosen precisely so that the Laplacian at backbone vertices is constantly equal to $3$, as will be shown in the proof of Lemma \ref{lem:outer-bound-stable}.  
The term $\Big\lfloor \frac{1+\sqrt{1+8f_R(x+1)}}{2}\Big\rfloor$
in \eqref{eq:odometer-backbone} determines the tooth heights in the support of our upper bound. To explain the intuition behind this expression, recall that for the divisible sandpile odometer  the tooth height is asymptotically  the square root of the odometer value along the backbone. This leads us to impose an analogous relation in our construction. However, because the tooth heights must be integer-valued, the choice requires some care.  We therefore choose the tooth heights $H_x$ so that the value of $f_R$ lies between two consecutive triangular numbers:
\[
\frac{1}{2}(H_x-1)H_x\leq f_R(|x|)<\frac{1}{2}(H_x+1)H_x.
\]
As the next lemma shows, this condition uniquely determines $H_x$ for every $x\in\mathbb{Z}$.

\begin{lemma}\label{lem:tooth-height-expression}
Let $m\in\mathbb{N}$, and let $g(m)\in\mathbb{N}$ be the unique integer satisfying
\[
\frac{1}{2}(g(m)-1)g(m)\leq m < \frac{1}{2}(g(m)+1)g(m).
\]
Then $g(m)$ is explicitly given by
\[
g(m)=\Big\lfloor \frac{1+\sqrt{1+8m}}{2}\Big\rfloor .
\]
Equivalently, the above inequality uniquely determines $g(m)$ through this formula.
\end{lemma}
\begin{proof}
    The first inequality $\frac{1}{2}(g(m)-1)g(m)\leq m$ implies $g(m)^2-g(m)\leq 2m$ which is equivalent to $\big(g(m)-\frac{1}{2}\big)^2\leq 2m+\frac{1}{4}$, which in turn gives $(2g(m)-1)^2\leq 8m+1$. Taking the square-root and rearranging yields $g(m)\leq \frac{1+\sqrt{1+8m}}{2}$. To obtain a corresponding lower bound on $g(m)$, we consider the second inequality $m < \frac{1}{2}(g(m)+1)g(m)$, which after rearranging is equivalent to $2m<g(m)^2+g(m).$ Completing the square gives $2m+\frac{1}{4}<\big(g(m)+\frac{1}{2}\big)^2$, thus $8m+1<(2g(m)+1)^2$. Rearranging implies $\sqrt{1+8m}<2g(m)+1$, which finally shows that $ \frac{1+\sqrt{1+8m}}{2}<g(m)+1$.
    Since $g(m)$ must be integer-valued, the upper and lower bound imply that $g(m)=\left\lfloor\frac{1+\sqrt{1+8m}}{2}\right\rfloor$.
\end{proof}
Thus, given the solution to \eqref{eq:odometer-backbone}, we can choose the corresponding tooth height for $x\in\Z$ as
\begin{align}\label{eq:tooth-heigh}\tag{TH}
    H_x=\Big\lfloor \frac{1+\sqrt{1+8f_R(|x|)}}{2}\Big\rfloor.
\end{align}
The final step in constructing the upper bound is to specify the locations of the holes on the individual teeth. First observe that, if the upper bound were required to have constant Laplacian along each tooth, a natural choice would be $\frac{1}{2}(H_x-|y|)(H_x-|y|+1)$.
Since we want exactly one hole on each tooth, we perturb this profile by a function of the form
$\max(0,y_x^{\mathrm{hole}}-|y|)$, whose Laplacian is identically zero except at the prescribed hole location $y_x^{\mathrm{hole}}$.  
To ensure that the resulting upper bound is well-defined, the hole location must be chosen so that, at $y=0$, the tooth profile agrees with the prescribed backbone value $f_R$. This forces the choice
\begin{align}\label{eq:hole}\tag{Hole}
y_x^{\text{hole}}=\frac{1}{2}(H_x+1)H_x-f_R(|x|),
\end{align}
for every $x\in\mathbb{Z}$. Moreover, by the definition of the tooth height $H_x$ in \eqref{eq:tooth-heigh}, we have $1\leq y_x^{\mathrm{hole}}\leq H_x$.
With these quantities now fixed, we are ready to define the upper bound.

\begin{lemma}\label{lem:outer-bound-stable}
Let $R\in\N$, and let $f_R$ be the solution of \eqref{eq:odometer-backbone} with boundary conditions $f_R(R)=1$ and $f_R(x)=0$ for all $x>R$. Fix the tooth heights $(H_x)_{x\in\Z}$ as in \eqref{eq:tooth-heigh} as well as the hole locations along the teeth $(y_x^{\text{hole}})_{x\in\Z}$ as in \eqref{eq:hole}. Define the function $v_R:\mathcal{C}_2\to\N_0$ as
\begin{align*}
    v_R(x,y)=\begin{cases}
        \frac{1}{2}(H_x-|y|)(H_x-|y|+1) -\max\bigl(0,y_x^{\mathrm{hole}}-|y|\bigr),&0\leq |y| \leq H_x\\
        0,&\text{else}
    \end{cases}.
\end{align*}
Then $v_R$ is integer-valued, and for   $(x,y)\in\mathcal{C}_2\setminus\{o\}$, it holds $-\Delta v_R(x,y)\leq \deg_{\mathcal{C}_2}(x,y)-1$.
\end{lemma}
\begin{proof}
    Since the boundary values of $f_R$ are integers and all terms in \eqref{eq:odometer-backbone} are integer valued, it follows that $f_R$ itself is integer valued. Consequently, since $v_R$ evaluates to $f_R$ on backbone vertices, $v_R$ takes integer values along the backbone. For the tooth vertices, observe that for every $y\in\mathbb{Z}$, the product $(H_x-|y|)(H_x-|y|+1)$ is always even. Hence the corresponding values of $v_R$ along the teeth are also integers.

    For the Laplacian along the teeth, observe that $v_R$ can be written as the sum of a function whose Laplacian is constantly $1$ for all $1\leq y\leq H_x$ and $0$ for $y>H_x$, minus a function whose Laplacian vanishes everywhere except at $y_x^{\mathrm{hole}}$, where it equals $1$. Therefore, the Laplacian along the teeth is always bounded above by $1$. It remains to consider the backbone vertices. Notice that
    $-\Delta v_R(R+1,0)=f_R(R)=1\leq 3$ and the same argument applies to the backbone vertex $(-R-1,0)$. For the Laplacian at $(R,0)$, we first use the recursion \eqref{eq:odometer-backbone} to compute $f_R(R-1)=7$. Since $H_R=2$ and $y_R^{\mathrm{hole}}=2$, it follows that $v_R(R,\pm 1)=0$. Hence, $-\Delta v_R(R,0)=7-4=3$, and the same calculation applies at $(-R,0)$. For the remaining backbone vertices $(x,0)\in\mathcal{C}_2$ with $x\neq 0$ and $|x|<R$, we obtain
    \begin{align*}
        -\Delta v_R(x,0)&=f_R(|x|-1)+f_R(|x|+1)+H_x(H_x-1)-2(y_x^{\text{hole}}-1)-4f_R(|x|)\\
        &=f_R(|x|-1)+f_R(|x|+1)+(H_x^2-H_x)-(H_x^2+H_x-2f_R(|x|)-2)-4f_R(|x|)\\
        &=f_R(|x|-1)+f_R(|x|+1)-2f_R(|x|)-2H_x+2\\
        &=\Big(f_R(|x|-1)+f_R(|x|+1)-2f_R(|x|)-2\Big\lfloor \frac{1+\sqrt{1+8f_R(|x|)}}{2}\Big\rfloor-1\Big)+3.
    \end{align*}
    Since $f_R$ solves \eqref{eq:odometer-backbone}, the first term on the right-hand side above is zero, thus $-\Delta v_R(x,0)=3$. For all other vertices it holds $\Delta v_R=0$, and this completes the proof.
\end{proof}

In order for the function constructed in Lemma \ref{lem:outer-bound-stable} to provide a suitable outer bound, we need to determine the asymptotic behavior of $f_R$ and of the corresponding tooth heights.
\begin{lemma}\label{lem:outer-bound-asymptotics}
    Let $R\in\mathbb{N}$, and let $f_R$ be the solution of \eqref{eq:odometer-backbone} with boundary conditions $f_R(R)=1$ and $f_R(x)=0$ for all $x>R$. Then, for every $0\leq x\leq R+1$, one has
\[
    \frac{(R+1-x)^4}{18}+\frac{(R+1-x)^2}{2}\leq f_R(x)\leq\frac{(R+1-x)^4}{18}+(R+1-x)^3.
\]
\end{lemma}
\begin{proof} Let us first reformulate the target function in a slightly different way, eliminating the dependence on $R$. Define a function $f:\mathbb{N}_0\to\mathbb{Z}$ by the initial conditions $f(0)=0$ and $f(1)=1$ and require that, for all $x>1$, it satisfies the recursion
\[
f(x+1)=2f(x)-f(x-1)+2\Big\lfloor \frac{1+\sqrt{1+8f(x)}}{2}\Big\rfloor+1.
\]
For  $x\in\N_0$ it holds $\sqrt{1+8x}\leq 2\Big\lfloor \frac{1+\sqrt{1+8x}}{2}\Big\rfloor+1\leq \sqrt{1+8x}+2$. Our strategy for obtaining an appropriate lower bound is to construct a function $L:\mathbb{N}_0\to\mathbb{R}$ such that $L(0)=0$, $L(1)\leq 1$
and, for every $x>1$, $\Delta L(x)\leq \sqrt{8L(x)+1}$. Letting $L(x)=x^4/18+x^2/2$, we get
$L(0)=0$, $L(1)=\frac{1}{18}+\frac{1}{2}<1$, and for the Laplacian at $x>1$ we obtain
    \begin{align*}
        \Delta L(x)^2&=\Big(\frac{2}{3}x^2+\frac{10}{9}\Big)^2=\frac{4}{9}x^4+\frac{40}{27}x^2+\frac{100}{81}\leq \frac{4}{9}x^4+4x^2+1=8L(x)+1.
    \end{align*}
    To complete the proof of the lower bound, we use the discrete comparison principle. Define $\varepsilon(x)=f(x)-L(x)$ and $\delta(x)=\varepsilon(x)-\varepsilon(x-1)$. Then $\varepsilon(0)=0$, $\varepsilon(1)=4/9$, and consequently $\delta(1)=4/9$. Now let $x>1$, and suppose that for every $z<x$ we have $\varepsilon(z)\geq 0$ and $\delta(z)\geq 0$. We first examine the Laplacian of the accumulated error:
\[
\Delta \varepsilon(x-1)=\Delta f(x-1)-\Delta L(x-1)\geq\sqrt{8f(x-1)+1}-\sqrt{8L(x-1)+1}\geq 0,
\]
where the last inequality follows from $f(x-1)\geq L(x-1)$. Hence, for the first-order difference sequence $\delta$, we obtain
\[
\delta(x)-\delta(x-1)=\bigl(\varepsilon(x)-\varepsilon(x-1)\bigr)-\bigl(\varepsilon(x-1)-\varepsilon(x-2)\bigr)=\Delta \varepsilon(x-1)\geq 0,
\]
which implies $\delta(x)\geq \delta(x-1)\geq 0$. 
By induction, it follows that $\delta(x)\geq 0$ for all $x\in\mathbb{N}_0$, and hence $f(x)\geq L(x)$ for every $x\in\mathbb{N}_0$. The lower bound asserted in the lemma is then obtained by shifting this inequality by $R$. The upper bound is proved in the same way, using instead the comparison function $U(x)=\frac{x^4}{18}+x^3$ for $x\in\mathbb{N}_0$.
\end{proof}

From Lemma \ref{lem:outer-bound-asymptotics}, we immediately derive the corresponding asymptotics for the tooth heights and for the Laplacian of the function $v_R$ at the origin, where $v_R$ is defined as in Lemma \ref{lem:outer-bound-stable}.

\begin{lemma}\label{lem:outer-bound-teeth-and-laplace}
For  $R\in\mathbb{N}$ and $v_R:\mathcal{C}_2\to\mathbb{Z}$ defined as in Lemma \ref{lem:outer-bound-stable} it holds
    $$\frac{4}{9}R^3+\frac{4}{3}R^2+\frac{32}{9}R+\frac{2}{3}\leq \Delta v_R(0,0)\leq \frac{4}{9}R^3+\frac{22}{3}R^2+\frac{158}{9}R+\frac{23}{3},$$
    and 
    \begin{align*}
    \Big\{(x,y)\in\mathcal{C}_2\;:\; \vert x\vert\leq R&\text{ and }\vert y\vert \leq \frac{1}{3}\big(R+1-|x|\big)^2\Big\}\subseteq \text{supp}(v_R)\\
    &\subseteq\Big\{(x,y)\in\mathcal{C}_2\;:\; \vert x\vert\leq R \text{ and } \vert y\vert \leq \frac{1}{3}\big(R+1-|x|\big)^2+3\big(R+1-|x|\big)+\frac{1}{2}\Big\}.\end{align*}
\end{lemma}
\begin{proof}
    We first analyze the asymptotic behavior of the tooth heights. Observe that
    \[
    \frac{-1+\sqrt{1+8f_R(|x|)}}{2}\leq H_x\leq \frac{1+\sqrt{1+8f_R(|x|)}}{2}.
    \]

  We may now insert the bounds from Lemma \ref{lem:outer-bound-asymptotics} to obtain explicit bounds on the tooth heights. For readability, set $z=R+1-|x|$. For the lower bound, we have
    \begin{align*}
        H_x&\geq\frac{-1+\sqrt{1+8\big(\frac{z^4}{18}+\frac{z^2}{2}\big)}}{2}\geq \frac{-1+\sqrt{\big(\frac{2}{3}z^2+1\big)^2}}{2}\geq \frac{z^2}{3}.
    \end{align*}
    The upper bound is obtained analogously:
    \begin{align*}
        H_x\leq \frac{1+\sqrt{1+8\big(\frac{z^4}{18}+z^3\big)}}{2}\leq \frac{1+\sqrt{\big(\frac{2}{3}z^2+6z\big)^2}}{2}=\frac{1}{3}z^2+3z+\frac{1}{2}.
    \end{align*}
    The two inequalities above prove the two inclusions of the support as stated in this lemma. We now move onto the Laplacian at the origin. We have
    \begin{align*}
        \sum_{x=1}^R \left(f_R({x-1})-2f_R(x)+f_R({x+1})\right)=\sum_{x=1}^R\left(2H_x+1\right).
    \end{align*}
    The left-hand side of this equation is a telescoping sum. Using $f_R(R)=1$ and $f_R(R+1)=0$, we obtain
    \[
    f_R(0)-f_R(1)=R+1+2\sum_{x=1}^R H_x.
    \]
Substituting this identity into the formula for the Laplacian at the origin yields
\[
-\Delta v_R(0,0)=-2\bigl(f_R(0)-f_R(1)\bigr)+2-2H_0=-2R-2H_0-4\sum_{x=1}^R H_x.
\]
Using the previously established upper and lower bounds for the tooth heights, we obtain
the first claim of the lemma.
\end{proof}
With these asymptotic estimates established, we are now ready to prove the outer bound on the limit shape for the single-source abelian sandpile stated in Theorem \ref{thm:main2}.
\begin{figure}[htb]
    \centering
    \begin{minipage}{0.49\textwidth}
        \centering
        \includegraphics[width=0.5\linewidth]{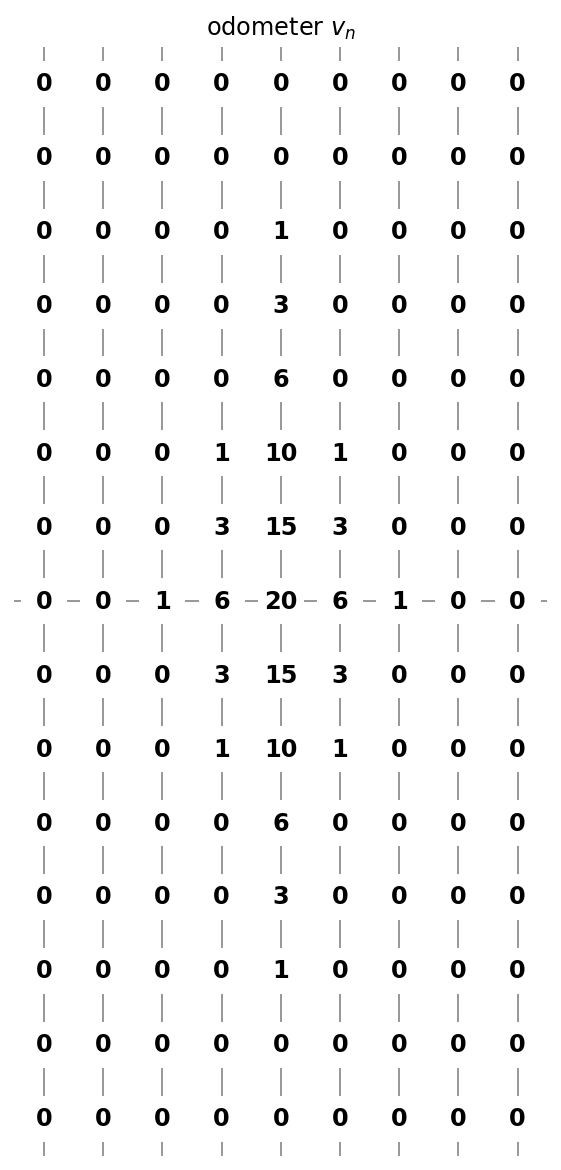}
    \end{minipage}
    \begin{minipage}{0.49\textwidth}
        \centering
        \includegraphics[width=0.5\linewidth]{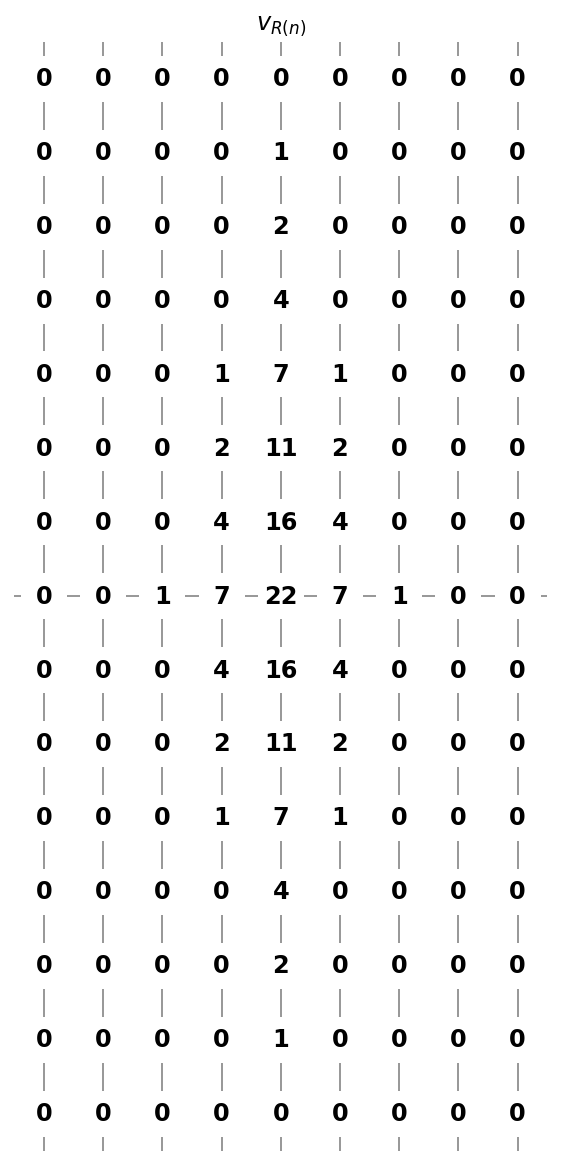}
    \end{minipage}
    \caption{On the left, the odometer resulting from the stabilization of the configuration $n\delta_o$ with $n=40$ is displayed. On the right, we display the function $v_{R(n)}$, with $R(n)=2$.}
    \label{fig:odometer}
\end{figure}
\begin{proposition}\label{prop:outer-bound}
     For $n\in\mathbb{N}$,  let $\mathcal{S}_n $ denote the set of vertices in $\mathcal{C}_2$ that topple during the stabilization of the sandpile configuration $n\delta_o$  of $n$ particles placed at $o=(0,0)$. Then, for every \(\varepsilon>0\), there exists \(N(\varepsilon)\in\mathbb{N}\) such that, for all \(n\ge N(\varepsilon)\), it holds
    $\mathcal{S}_n\subseteq \mathcal{B}_{(1+\varepsilon)n}$.
\end{proposition}
\begin{proof}
Stabilize the sandpile $n\delta_o$ according to the abelian sandpile toppling procedure, and denote the resulting odometer function by $v_n$. Define $R(n)$ by
\[
R(n)=\inf\bigl\{R\in\mathbb{N}:\ \Delta v_R(o)\geq n\bigr\},
\]
where, for each $R\in\mathbb{N}$, the function $v_R$ is the one introduced in Lemma \ref{lem:outer-bound-stable}. Then for all $x\in\mathcal{C}_2$ 
    $$n\delta_o-\Delta v_{R(n)}(x)< \deg(x),$$
    and thus by the least action principle $v_{R(n)}\geq v_n.$  By the asymptotic estimate for the Laplacian of $v_R$ at $o=(0,0)$ in Lemma \ref{lem:outer-bound-teeth-and-laplace}, for every $\varepsilon>0$ there exists $N(\varepsilon)\in\mathbb{N}$ such that, for all $n\geq N(\varepsilon)$,
    \[
    R(n)\leq \left(1+\frac{\varepsilon}{2}\right)^{1/3} \Big(\frac{3}{2}\Big)^{2/3}n^{1/3}.
    \]
    This together with the second claim from Lemma \ref{lem:outer-bound-teeth-and-laplace}
    implies that for $n$ sufficiently large
    $$\mathcal{S}_n\subseteq \Big\{ (x,y)\in\mathcal{C}_2\; :\; |x|\leq \Big((1+\varepsilon)\frac{9n}{4}\Big)^{1/3} \text{ and } |y|\leq \frac{1}{3}\Big(\Big((1+\varepsilon)\frac{9n}{4}\Big)^{1/3}-|x|\Big)^2\Big\}=\mathcal{B}_{(1+\varepsilon)n}.$$
\end{proof}
In Figure \ref{fig:odometer}, we display the odometer function $v_{40}$ together with the comparison function $v_{R(40)}$, which is an upper bound for the odometer, as established in the proof of Proposition \ref{prop:outer-bound}.

\begin{proof}[Proof of Theorem \ref{thm:main2}]
    The statement follows from Proposition \ref{prop:inner-bound} and \ref{prop:outer-bound}.
\end{proof}

\section{Concluding remarks and open questions}

The results of this paper motivate several directions for future research, some of which we highlight below.

\textbf{Conditions for limit shape universality.} Theorem \ref{thm:main2} shows that the single-source limit shape of the abelian sandpile coincides with the limit shape of the divisible sandpile, IDLA, and rotor-router aggregation \cite{IDLA-comb-2012}. The same type of limit shape universality has also been proved for the infinite double-sided Sierpi\'nski gasket graph \cite{CK20, IDLA-gasket-2020, Div-sand-gasket-2019}. 
Interestingly, although the abelian sandpile limit shape on $\mathbb{Z}^2$ and in higher dimensional Euclidean lattices remains unproved, it is widely believed that on such lattices the abelian sandpile limit shape is not a Euclidean ball. Thus, it is expected to differ from the limit shape of the divisible sandpile, IDLA, and rotor-router aggregation. This raises the question of which structural properties are shared by the comb lattice and the Sierpi\'nski gasket, causing all four models to exhibit the same limit shape, while on $\mathbb{Z}^2$ the abelian sandpile appears to behave differently. We believe that the answer is closely connected to the behavior of loop-erased random walk on the underlying state space.

\begin{question}
    Given a countably infinite graph $G$, determine necessary and sufficient conditions under which the abelian sandpile, divisible sandpile, IDLA, and rotor-router aggregation all exhibit the same limit shape.
\end{question}

\begin{figure}
    \centering
    \begin{minipage}{0.42\textwidth}
        \centering
        \includegraphics[height=8.5cm]{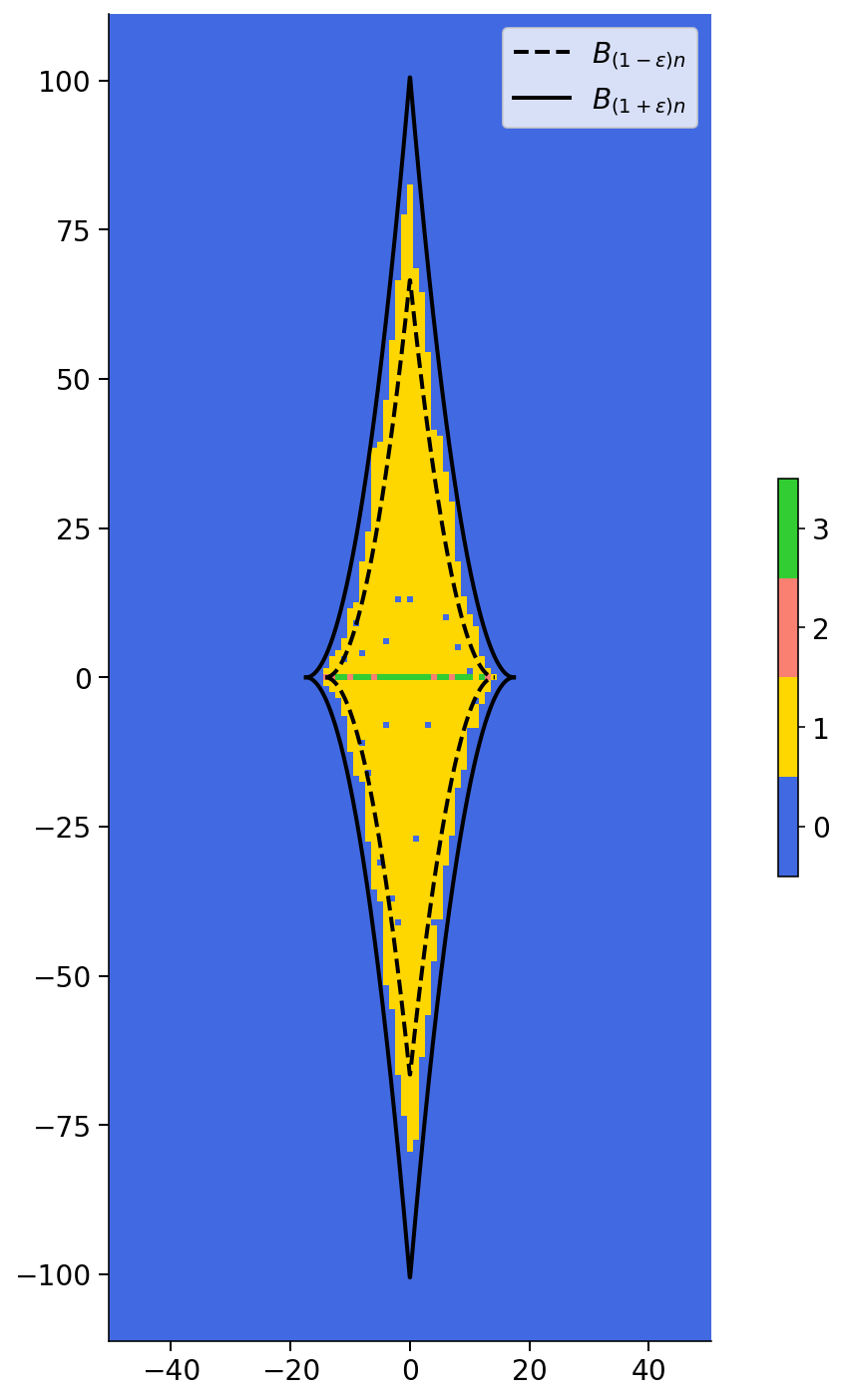}
    \end{minipage}
    \hfill
    \begin{minipage}{0.42\textwidth}
        \centering
        \includegraphics[height=8.5cm]{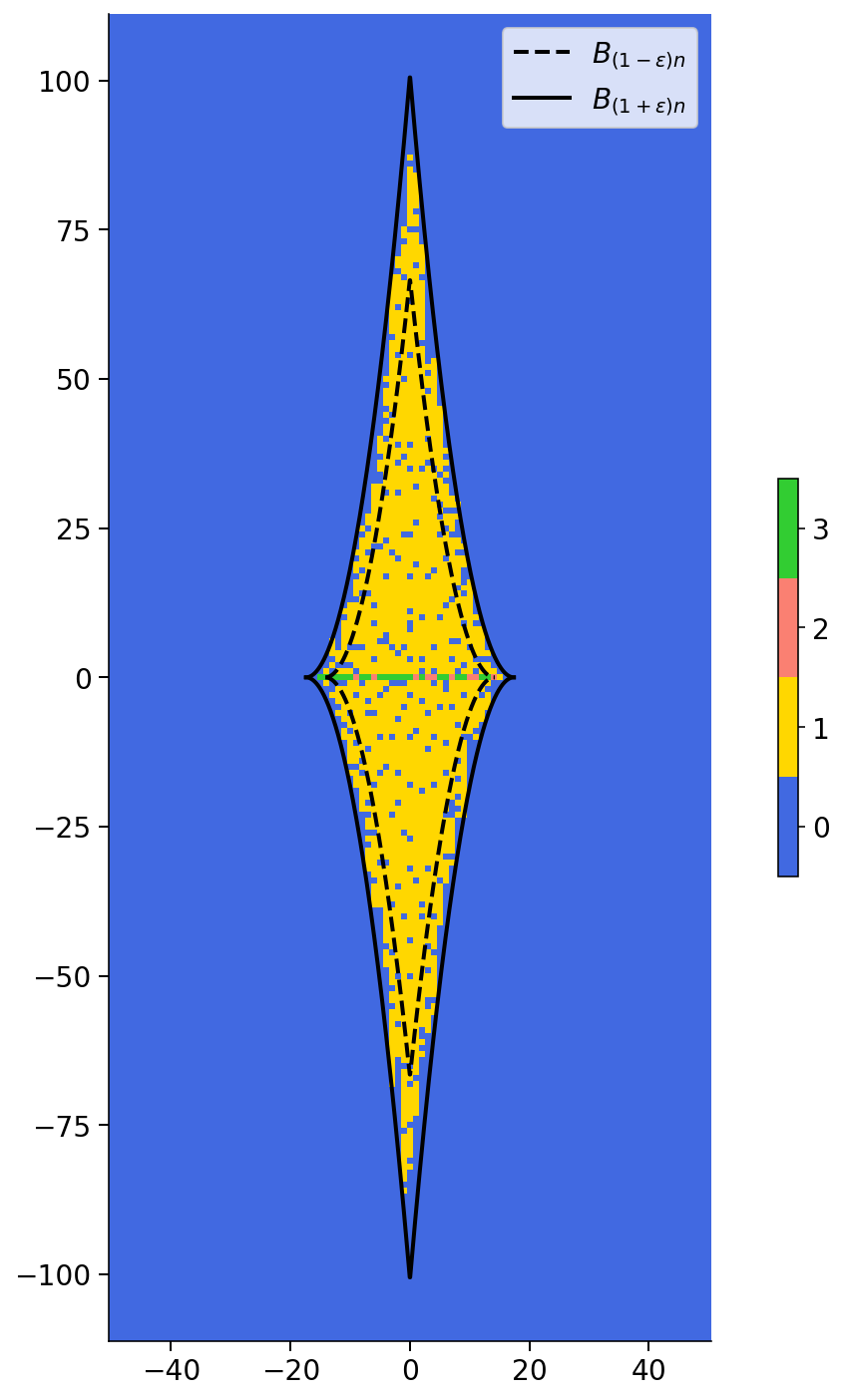}
    \end{minipage}
    \caption{Limit shape of the single-source stochastic sandpile on the comb, obtained by placing $n$ particles at $o=(0,0)$ and stabilizing. The dashed lines indicate the boundary of $\mathcal{B}_{(1-\varepsilon)n}$, while the solid lines indicate the boundary of $\mathcal{B}_{(1+\varepsilon)n}$. In both figures, $\varepsilon=0.3$ and $n=1800$. On the left, unstable vertices topple according to $p$-topplings with $p=1/2$ and on the right, each unstable vertex sends two particles to neighbors chosen independently and uniformly.}
    \label{fig:simulation_stochastic_model}
\end{figure}

\textbf{Single-source limit shape for stochastic sandpiles.} 
Stochastic sandpiles generalize the abelian sandpile model by allowing the toppling operations themselves to be random. Two examples of such randomized toppling mechanisms are the $p$-topplings \cite{p-topplings}, in which, whenever a vertex topples, each incident edge is chosen independently with probability $p\in[0,1]$ to receive one particle, and the Manna model \cite{Manna1991TwoState}, in which each toppled particle independently performs one random walk step. 
We expect that, for both of these models, the single-source limit shape on the comb lattice coincides with that of the abelian sandpile, as suggested by the simulations in Figure \ref{fig:simulation_stochastic_model}. Let $\mathcal{S}_n^{p}$ denote the set of vertices that topple during the stabilization of $n\delta_o$ under the $p$-toppling dynamics, and let $\mathcal{S}_n^{\mathrm{Manna}}$ denote the corresponding set of toppled vertices for the Manna model.
\begin{question}
    Show that, for every $\varepsilon>0$, the following holds with probability $1$ as $n\to\infty$:
    $$\mathcal{B}_{(1-\varepsilon)n}\subseteq \mathcal{S}_n^p\subseteq \mathcal{B}_{(1+\varepsilon)n}\hspace{1cm}\text{and}\hspace{1cm}\mathcal{B}_{(1-\varepsilon)n}\subseteq \mathcal{S}_n^{\text{Manna}}\subseteq \mathcal{B}_{(1+\varepsilon)n}.$$
\end{question}

\textbf{Cutoff for the sandpile Markov chain on the comb.} 
In Theorem \ref{thm:main}, we study the stationary distribution of the sandpile Markov chain on finite boxes centered at the origin. A natural further question is to determine how rapidly the sandpile chain converges to this stationary distribution, that is, to obtain bounds on its mixing time. It would also be interesting to investigate whether this convergence occurs abruptly around the mixing time, in other words, whether the sandpile Markov chain exhibits cutoff on the comb lattice. The cutoff phenomenon for the sandpile Markov chain has been established on $\mathbb{Z}^2$ \cite{HJL19} and on the complete graph \cite{JLP19}, and it would be interesting to identify further classes of graphs for which cutoff occurs.

\begin{question}
    Consider the Abelian sandpile Markov chain on a box of size $(2n+1)\times(2n+1)$ in the comb lattice $\mathcal{C}_2$, and let $t_{\mathrm{mix}}^{(n)}(\varepsilon)$ denote its mixing time. Does it hold that $\lim_{n\to\infty}\frac{t_{\mathrm{mix}}^{(n)}(\varepsilon)}{t_{\mathrm{mix}}^{(n)}(1-\varepsilon)}=1$ for every $\varepsilon\in(0,1)$?
\end{question}

\textbf{Use of artificial intelligence.} The authors developed the main ideas of this work and directed
its mathematical development. Throughout the project, the authors used large language models
as research assistants: to explore specific questions arising in the proofs, to check some arguments
and the cited results. The authors take responsibility for the mathematical content and presentation of the paper.

\textbf{Acknowledgments.}
This research was funded in part by the Austrian Science Fund (FWF) [10.55776/\allowbreak{}PAT3123425]. For open access purposes, the authors have applied a CC BY public copyright license to any author-accepted manuscript version arising from this submission.

\bibliographystyle{alpha}

\textsc{Robin Kaiser}, Departement of Mathematics, CIT, Technische Universität München, Boltzmannstr. 3, D-85748 Garching bei München, Germany. \texttt{ro.kaiser@tum.de}

\textsc{Ecaterina Sava-Huss}, Universität Innsbruck, Institut für Mathematik, Technikerstraße 13,
A-6020 Innsbruck, Austria. \texttt{Ecaterina.Sava-Huss@uibk.ac.at}

\textsc{Julia Überbacher}, Universität Innsbruck, Institut für Mathematik, Technikerstraße 13,
A-6020 Innsbruck, Austria. \texttt{Julia.Ueberbacher@uibk.ac.at}
\end{document}